\documentclass[10pt,regno,numberwithinsect]{amsart}
\usepackage{amsmath,amsthm,amsfonts,amssymb,amscd,euscript,graphicx,overpic,color,bbm}
\pdfoutput=1

\usepackage{chngcntr}
\usepackage{apptools}
\AtAppendix{\counterwithin{applemma}{section}}

\newtheorem{theorem}{Theorem}[section]

\newtheorem{lemma}[theorem]{Lemma}

\newtheorem{corollary}[theorem]{Corollary}
\newtheorem{proposition}[theorem]{Proposition}
\newtheorem{claim}{Claim}
\newtheorem*{claim*}{Claim}

\theoremstyle{definition}
\newtheorem{definition}[theorem]{Definition}
\newtheorem*{remark*}{Remark}
\newtheorem{remark}[theorem]{Remark}

\numberwithin{equation}{section}

\newcommand{\eqdef}{\stackrel{\scriptscriptstyle\rm def}{=}}

\def\c{{\rm c}}

\def\u{{\rm u}}

\DeclareMathOperator{\dist}{Dist}

\DeclareMathOperator{\Per}{Per}

\DeclareMathOperator{\card}{card}

\def\bN{\mathbb{N}}
\def\bZ{\mathbb{Z}}
\def\bS{\mathbb{S}}
\def\bE{\mathbb{E}}
\def\bC{\mathbb{C}}

\def\bR{\mathbb{R}}

\def\cO{\EuScript{O}}
\def\cN{\EuScript{N}}

\def\cM{\EuScript{M}}

\def\NN {{\mathbb N}}

\DeclareMathSymbol{\varnothing}{\mathord}{AMSb}{"3F}
\renewcommand{\emptyset}{\varnothing}
\title[Structure of the space of ergodic measures]{The structure of the space of ergodic measures of transitive partially hyperbolic sets}
\author[L.~J.~D\'iaz]{L. J. D\'\i az}
\address{Departamento de Matem\'atica PUC-Rio, Marqu\^es de S\~ao Vicente 225, G\'avea, Rio de Janeiro 225453-900, Brazil}
\email{lodiaz@mat.puc-rio.br}
\author[K.~Gelfert]{K.~Gelfert}
\address{Instituto de Matem\'atica Universidade Federal do Rio de Janeiro, Av. Athos da Silveira Ramos 149, Cidade Universit\'aria - Ilha do Fund\~ao, Rio de Janeiro 21945-909,  Brazil}
\email{gelfert@im.ufrj.br}
\author[T.~Marcarini]{T.~Marcarini}
\address{Departamento de Matem\'atica Universidade Federal do Esp\'irito Santo, Av. Fernando Ferrari 514, Campus de Goiabeiras,  Vit\'oria 29075-910,  Brazil}
\email{tiane.pinto@ufes.br}
\author[M.~Rams]{M. Rams} \address{Institute of Mathematics, Polish Academy of Sciences, ul. \'{S}niadeckich 8,  00-656 Warszawa, Poland}
\email{rams@impan.pl}

\thanks{This research has been supported  [in part]  by CNE-Faperj, CNPq-grants (Brazil), and National Science Centre grant 2014/13/B/ST1/01033 (Poland). The authors acknowledge the hospitality of IMPAN, IM-UFRJ, and PUC-Rio.}

\begin{document}

\begin{abstract}
	We provide examples of  transitive partially hyperbolic dynamics (specific but paradigmatic examples of homoclinic classes) which blend  different types of hyperbolicity in the one-dimensional center direction. 
	These homoclinic classes have two disjoint parts: an ``exposed'' piece which is poorly homoclinically related with the rest and a ``core'' with rich homoclinic relations.
	There is an associated natural division of the space of ergodic measures which are either supported on the exposed piece or on the core.  We describe the topology of these two parts and show that they glue along nonhyperbolic measures. 
	
	Measures of maximal entropy are discussed in more detail. We present examples where the measure of maximal entropy is nonhyperbolic. We also present examples where the  measure of maximal entropy is unique and nonhyperbolic, however in this case the dynamics is nontransitive. 
\end{abstract}


\maketitle

%------------------------------------------------------------------------------------------------------
\section{Introduction}
%------------------------------------------------------------------------------------------------------

An important task in ergodic theory is to describe the topology of the space of invariant and/or ergodic measures which are supported on a given invariant set. Here  in many cases the weak$\ast$ topology is considered, though one also studies convergence in the weak$\ast$ topology and entropy.
Recently there happened a certain revival of this type of problems in the context of nonhyperbolic dynamical systems \cite{GelKwi:,GorPes:17,DiaGelRam:17,BocBonGel:}, most of them revisiting the pioneering work of Sigmund on topological dynamical systems satisfying  the specification property \cite{Sig:74,Sig:77}.

For a general continuous map $F$ on a metric space $\Lambda$, consider the set of $F$-invariant Borel probability measures $\cM(\Lambda)$ and denote by $\cM_{\rm erg}(\Lambda)$ the subset of ergodic ones. If $\Lambda$ is compact then $\cM=\cM(\Lambda)$ is a Choquet simplex whose extremal elements are the ergodic measures.  Density of ergodic measures in $\cM$ implies that  either $\cM$  is a singleton (when $F$ is uniquely ergodic) or a nontrivial simplex whose extreme points are dense. In the latter case, it is \emph{the} so-called \emph{Poulsen simplex} and by~\cite{LinOlsSte:78} has immediately a number of further strong properties such as arcwise connectedness.
Sigmund~\cite{Sig:74,Sig:77} addressed first the questions on the density of ergodic measures and also the properties of generic invariant measures. He showed that for a map $F$ satisfying the so-called \emph{periodic specification property} the periodic measures (and thus the ergodic ones) are dense in $\cM$. Here a measure is \emph{periodic} if it is the invariant probability measure supported on a periodic orbit. 
Moreover, the sets of  ergodic measures and of measures with entropy zero are both residual in $\cM$. For an updated discussion and more references, see \cite{GelKwi:}. 

Observe that Sigmund's results~\cite{Sig:74,Sig:77}  immediately apply to any basic set of a smooth Axiom A diffeomorphism. In a (more) general context, to address the general question if the space $\cM$ has dense extreme points or at least is connected, some natural requirements are to be satisfied. An important one  is certainly topological transitivity,  which is however far from being sufficient as for example there exist minimal systems with exactly two ergodic measures. 

Nowadays arguments which provide the connectedness of $\cM$ are largely based on the approximation of invariant measures by periodic measures or Markov ergodic measures supported on horseshoes (a specific type of basic set). This demands that the periodic orbits involved are hyperbolic and somehow dynamically related among themselves. A natural relation introduced by Newhouse \cite{New:80}, and used in this context, is the \emph{homoclinic relation}, that is, the un-/stable invariant sets of these orbits intersect cyclically and transversally. 

A natural strategy is to study the components of the space of measures which each are candidate to correspond to one of the ``elementary'' undecomposable pieces of the dynamics. One of the possibilities to define properly what is meant by elementary is the \emph{homoclinic class}, that is,   the closure of the hyperbolic periodic orbits which are homoclinically related to the orbit of a hyperbolic periodic point $P$ and denoted by $H(P)$. Note that one of the fundamental properties is that the dynamics on each class is topologically transitive.
Basic sets of the hyperbolic theory mentioned above are the simplest examples of homoclinic classes.  
 
Notice that, when defining a homoclinic class, taking the closure can incorporate other orbits which are dynamically related but which are of different type of hyperbolicity. 
In this way, homoclinic classes may fail to be hyperbolic, contain saddles of different types of hyperbolicity (different $\u$-index, that is, dimension of unstable manifold), exhibit internal cycles, and support nonhyperbolic measures (also with positive entropy). 
Homoclinic classes of periodic points of different indices may even coincide.
Furthermore, there are examples where a homoclinic class $H(P)$ of a periodic point $P$ properly contains another class $H(P')$ of a periodic point $P'$ of the same index as $P$. Note that this precisely occurs if $P'\in H(P)$ was not homoclinically related to $P$. One sometimes refers to $H(P')$ as an \emph{exposed piece} of $H(P)$ \cite{DiaGelRam:14}. This type of phenomenon is a key ingredient in this paper. This gives only a rough idea what complicated structure these classes may have, see also \cite[Chapter 10.4]{BonDiaVia:05} for a more complete discussion. 

To be more precise for the following, we say that an ergodic measure $\mu$ is \emph{hyperbolic} if its Lyapunov exponents are nonzero. Moreover, almost all points have the same number $u=u(\mu)$ of positive Lyapunov exponents and we call this number $u$ the {\emph{$\u$-index of $\mu$}} (analogously to hyperbolic periodic measures above).  Given $u$, we denote by denote by $\cM_{{\rm erg},u}$ the set of ergodic measures of $\u$-index $u$. Note that in general one may have $\cM_{{\rm erg},u}(H(P))\ne\emptyset$ for several values of $u$. 

For the following  let us study the topological structure of $\cM_{{\rm erg},u}(H(P))$ for $u$ being the index of $P$.  Assuming that $H(P)$  is locally maximal and that all the saddles of index $u$ are homoclinically related, in \cite{GorPes:17} it is shown  that $\cM_{{\rm erg},u}(H(P))$ is path connected with periodic measures being dense and that its closure is a Poulsen simplex. Note that $\cM_{{\rm erg},u}(H(P))$ may only capture some part of $\cM_{{\rm erg}}(H(P))$.
Indeed this occurs when $H(P)$ contains saddles of different indices.
Still in this context, assume now that there coexists a saddle $Q$ of index $v\ne u$  and having the property that $H(Q)\subset H(P)$ (in an extreme case, these classes can even coincide as sets) and assume that all the saddles of index $v$ in $H(Q)$ are homoclinically related with $Q$ and consider $\cM_{{\rm erg},v}(H(Q))$. Though the interrelation between $\cM_{{\rm erg},u}(H(P))$ and $\cM_{{\rm erg},v}(H(Q))$ is not addressed in \cite{GorPes:17}, note that, by the very definition, they  are disjoint. Nevertheless, their closures may intersect or may not. Indeed, the space $\cM_{\rm erg}(H(P))$ may be connected or may not. To address this point is precisely the goal of this paper.

We introduce a class of examples of saddles $P$ and $Q$ of different indices whose homoclinic classes coincide $H(P)=H(Q)=\Lambda$
such that $\Lambda$ is the disjoint union of two invariant sets $\Lambda_{\mathrm{ex}}$ (a compact set that is a topological horseshoe) and 
$\Lambda_{\mathrm{core}}$. Moreover, these sets satisfy the following properties:
(i) $P, Q\in \Lambda_{\mathrm{core}}$ and the closure of $\Lambda_{\mathrm{core}}$ is the whole homoclinic class,
(ii) every pair of saddles of the same index in $\Lambda_{\mathrm{core}}$ (respectively, $\Lambda_{\mathrm{ex}}$) are homoclinically related, and
(iii) no saddle  in $\Lambda_{\rm core}$ is homoclinically related to any one in $\Lambda_{\rm ex}$.
We refer to $\Lambda_{\mathrm{ex}}$ as the \emph{exposed piece} of $\Lambda=H(P)=H(Q)$ and to $\Lambda_{\rm core}$ as its \emph{core}. We study the space $\cM_{{\rm erg}}(\Lambda)$ and show that it has an interesting topological structure: the set $\cM_{{\rm erg}}(\Lambda)$  has three pairwise disjoint parts 
$\cM_{{\rm erg},u}(\Lambda)$, $\cM_{{\rm erg},v}(\Lambda)$, $v=u+1$ and $u,v$ are
the indices of $P$ and $Q$,  and $\cM_{{\rm erg}}(\Lambda_{\mathrm{ex}})$, 
such that
$$
\cM_{{\rm erg}}(\Lambda)= \cM_{{\rm erg},u}(\Lambda)
\cup \cM_{{\rm erg},u}(\Lambda)  \cup 
\cM_{{\rm erg}}(\Lambda_{\mathrm{ex}})
\cup \cM_{{\rm erg, nhyp}}(\Lambda)
$$
where $\cM_{{\rm erg, nhyp}}(\Lambda)$ is the set of of nonhyperbolic ergodic measures of $\Lambda$. 
Note that $\cM_{{\rm erg}}(\Lambda_{\mathrm{ex}})$ and $\cM_{{\rm erg, nhyp}}(\Lambda)$ may intersect.
Moreover, the sets 
$\mathrm{closure}(\cM_{{\rm erg},u}(\Lambda))$, 
$\mathrm{closure}(\cM_{{\rm erg},v}(\Lambda))$,
and $\mathrm{closure}(\cM_{{\rm erg}}(\Lambda_{\mathrm{ex}}))$,  
are Poulsen simplices  whose intersection is contained in
$\cM_{{\rm erg, nhyp}}(\Lambda)$, see
Theorem~\ref{thepro:measurespace}.
Figure~\ref{Fig:2} below illustrates the interrelation between the measure space components.

\begin{figure}[h] 
\begin{overpic}[scale=.6,bb=0 0 200 200]{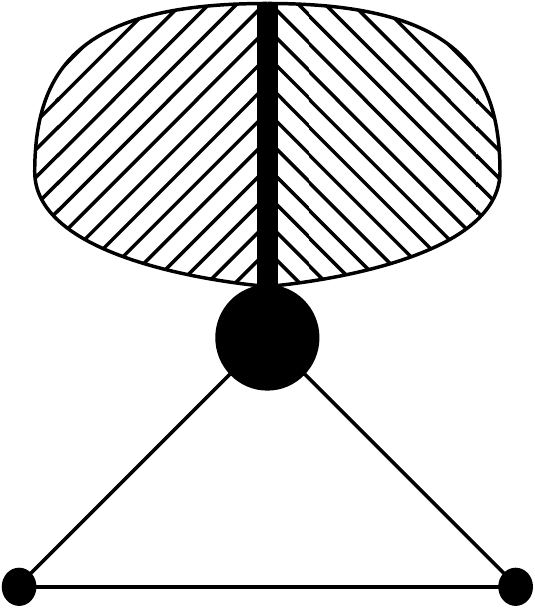}
 \put(80,85){{\small{$\cM_{{\rm erg},u}(\Lambda)$}}}
  \put(-38,85){{\small{$\cM_{{\rm erg},u+1}(\Lambda)$}}}
   \put(55,42){{\small{$\cM_{{\rm erg,nhyp}}(\Lambda)$}}}
    \put(25,15){{\small{$\cM_{{\rm erg}}(\Lambda_{\mathrm{ex}})$}}}
     \put(-13,-1){\small{$\delta_{P_{\mathrm{ex}}}$ }}
        \put(89,-1){\small{$\delta_{Q_{\mathrm{ex}}}$ }}
\end{overpic}
\caption{The space $\cM_{{\rm erg}}(\Lambda)$}
\label{Fig:2}
\end{figure}

Let us say a few additional words about the topological structure of the set $\Lambda=H(P)=H(Q)$. There are two exposed saddles
$P_{\mathrm{ex}}, Q_{\mathrm{ex}} \in \Lambda_{\mathrm{ex}}$ of the same indices such as $P$ and $Q$, respectively, which are involved in 
a {\emph{heterodimensional cycle}} (i.e., the invariant sets of these saddles meet cyclically),
 Indeed, the intersections of these invariant  sets give rise to the exposed
piece of dynamics that satisfy $\Lambda_{\mathrm{ex}}=H(P_{\mathrm{ex}})=H(Q_{\mathrm{ex}})\subsetneq\Lambda$. We are aware that on one hand this is a quite specific dynamical configuration, on the  other hand it provides paradigmatic examples. We also observe that this dynamical configuration resembles in some aspects the so-called Bowen eye (a two dimensional vector field having two saddle singularities involved in a double saddle connection) in \cite{Gau:92,Tak:94} and the examples due to Kan of intermingled basins of attractions (where an important property is that the boundary of an annulus is preserved) \cite{Kan:84}.
Finally, if we considered systems satisfying some boundary conditions or preserving a boundary, the conditions considered are quite general.

A particular emphasize is given to the measures of maximal entropy. In some cases, In some cases, these measures can be nonhyperbolic. We give a (non-transitive) example where the unique measure of maximal entropy is nonhyperbolic. 

Finally, we state of results for step skew products (these examples have differentiable realizations as partially hyperbolic sets with one dimensional central direction) and throughout the paper we do not aim  generality, on the contrary our goal is to make the construction in the simplest setting emphasizing the key ingredients behind the constructions.

This paper is organized as follows. In Section~\ref{s.settingandstatement} we state precisely our setting and our examples and state our main results.
In Section~\ref{sec:twinsetc}
 we study the ``symmetries" between certain measures and investigate entropy.  In Section \ref{sec:measures} we study the approximation of
 ``boundary measures". In Section~\ref{s.core} we study the measures supported in $\Lambda_{\mathrm{core}}$. In Appendix \ref{a.apendix} we provide details on
 transitivity and homoclinic relations in our examples and we analyze  examples with nonhyperbolic measures of maximal entropy. 

%------------------------------------------------------------------------------------------------------
\section{Setting and statement of results} 
\label{s.settingandstatement}
%------------------------------------------------------------------------------------------------------

We  now define precisely the dynamics that we will study.
Consider $C^1$ diffeomorphisms $f_0,f_1\colon[0,1]\to[0,1]$  satisfying the following properties:
\begin{enumerate}
\item [(H1)] 
The map $f_0$ has (exactly) two fixed points $f_0(0)=0$ and $f_0(1)=1$,
satisfies  $f_0'(0)=\beta>1$ and  $f_0'(1)=\lambda\in (0,1)$.
\\[-0.4cm]
\item [(H2)] 
The map $f_1$ has negative derivative and satisfies $f_1(0)=1$ and $f_1(1)=0$.
\end{enumerate}
The simplest (and also paradigmatic) example occurs when $f_1(x)=1-x$.

\begin{figure}[h] 
\begin{overpic}[scale=.35,bb=0 0 330 330]{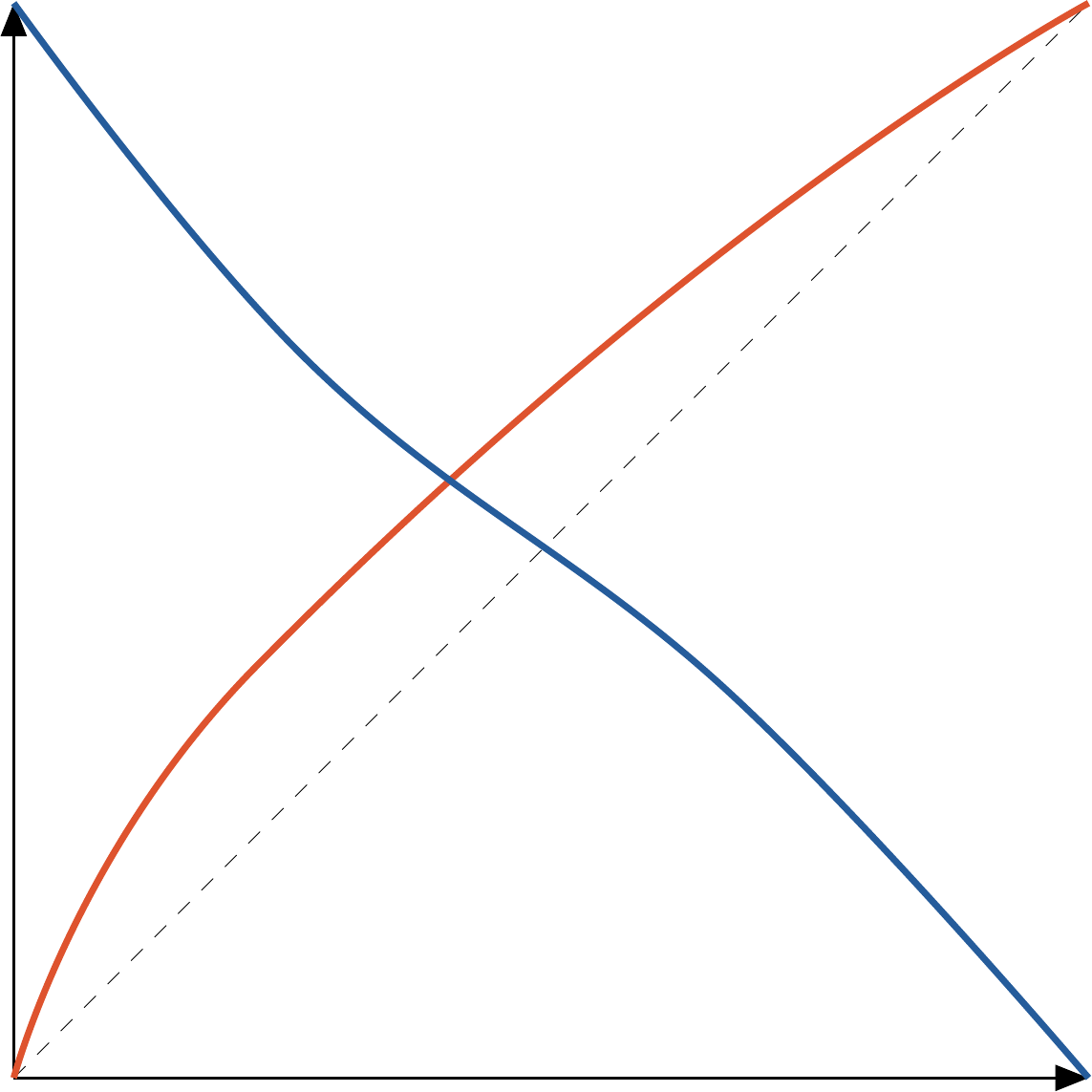}
 \put(60,85){$f_0$}
 \put(80,30){$f_1$}
\end{overpic}
\caption{Fiber maps in~\eqref{eq:defF}}
\label{Fig:1}
\end{figure}

Let $\sigma \colon \Sigma_2 \to \Sigma_2$ be the standard shift map on the shift space $\Sigma_2=\{0,1\}^{\bZ}$ of two-sided sequences, endowed with the usual metric. Consider the one step-skew product map $F$ associated to $\sigma$ and the maps $f_0$ and $f_1$ defined by
\begin{equation}\label{eq:defF}
	F\colon \Sigma_2\times [0,1]\to\Sigma_2 \times [0,1],\quad
	( \xi, x) \mapsto \big( \sigma(\xi) , f_{\xi_0}(x) \big).
\end{equation}
We consider the following $F$-invariant  subsets of $\Lambda\eqdef\Sigma_2\times [0,1]$
\begin{equation}\label{eq:defLambda}
	\Lambda_{\rm{ex}}\eqdef \Sigma_2\times \{0,1\},
	\quad
	\Lambda_{\rm{core}} 
	\eqdef  ( \Sigma_2\times [0,1])\setminus \Lambda_{\rm{exc}}
	=\Sigma_2\times(0,1).
\end{equation}
We say that $\Lambda_{\rm{ex}}$ is the \emph{exposed piece} of
$\Sigma_2\times [0,1]$ and that $\Lambda_{\rm{core}}$ is the \emph{core} of 
$\Sigma_2\times [0,1]$ (these denominations are justified below). Note that 
$\Lambda_{\rm{ex}}$ is a closed while $\Lambda_{\rm{core}}$ is not. Moreover, $F|_{\Lambda_{\rm ex}}$ is topologically transitive.
In fact, $F|_{\Lambda_{\rm ex}}$ is conjugate to a subshift of finite type, one may think this dynamical system as a horseshoe in a ``plane'', in that plane any pair of   saddles are ``homoclinically related''.

\begin{remark}[{Topological dynamics on $\Sigma_2\times[0,1]$}]
	While the dynamics in $\Lambda_{\rm ex}$ is completely characterized, in our quite general setting very few can be said about the dynamics of $F$ in $\Lambda_{\rm core}$. The most interesting case certainly occurs when $F|_{\Lambda_{\rm core}}$ is topologically transitive. Below we will see more specific examples where this transitivity indeed holds and, moreover, hyperbolic periodic orbits of positive and negative Lyapunov exponent are both dense in $\Sigma_2\times[0,1]$ and homoclinically related. We will see that nevertheless the measure space $\cM(\Lambda_{\rm ex})$ is ``semi-detached'' from $\cM(\Lambda_{\rm core})$. \end{remark}

Consider now more specific hypotheses on the $C^1$ diffeomorphisms interval maps $f_0,f_1\colon[0,1]\to[0,1]$:
\begin{enumerate}
\item [(H2')] $f_1(x)=1-x$.
\item [(H3)] The derivative $f_0^\prime$ is decreasing. Considering the point $c\in (0,1)$ defined by the condition $f_0^\prime (c)=1$, it holds
$
	f_1 \circ f_0^2(c)> f_0^2(c).
$
\item [(H4)] The numbers $\lambda$ and $\beta$ given in (H1) satisfy
\[
\varkappa\eqdef
\frac{\lambda^2\, (1-\lambda)}{\beta\, (\beta-1)}>1.
\]
\end{enumerate}
Observe that for fixed $\lambda$, the inequality in (H4) holds whenever $\beta$ is close enough to $1$.

\begin{proposition}\label{pro:22}
	Assume that $F$ defined in~\eqref{eq:defF} satisfies the hypotheses (H1), (H2'), (H3), and (H4). Then $F$ is topologically transitive. Moreover, every pair of fiber expanding hyperbolic periodic orbits and every pair of fiber contracting hyperbolic periodic orbits in $\Lambda_{\rm core}$ are homoclinically related, respectively.  
\end{proposition}

\begin{remark}[Discussion of hypotheses]
	Homoclinic relations for skew products are recalled in Appendix \ref{a.apendix}, where also the above proposition is proved. Condition (H4) will provide so-called expanding itineraries which in turn imply the homoclinic relations and their density for expanding points, while condition (H3) takes care of so-called contracting itineraries and the corresponding homoclinic relations. Thus, we conclude transitivity. The proof follows largely blender-like standard arguments used in~\cite{DiaGel:12}. Condition (H2') is only used for simplicity and also to follow more closely the model in \cite{DiaMar:15}. The key facts remain true assuming only (H2), in particular we never use the fact that for (H2') the map $f_1$ is an involution.

We observe that (H3) and (H4) demand a certain ``asymmetry" of the fiber map $f_0$. In Section~\ref{app:bad} we will provide a ``symmetric'' example which satisfies (H1) and for which the associated skew product fails to be transitive and its only measure of maximal entropy is nonhyperbolic and supported on $\Lambda_{\rm ex}$.
\end{remark}

\begin{remark}[{Examples in $\Sigma_3\times \mathbb{S}^1$ and $\Sigma_2\times \mathbb{S}^1$}]\label{rem:crcle}
We can produce a transitive example in  $\Sigma_3\times \mathbb{S}^1$ with properties analogous to the one in Proposition~\ref{pro:22} as follows.
Obtain $\mathbb{S}^1$ by identifying the boundary points of $[0,2]$. Define $g_0,g_1,g_2\colon \mathbb{S}^1\to \mathbb{S}^1$ as follows
\begin{itemize}
\item
 $g_0(x)=f_0(x)$ if $x\in [0,1]$ and $g_0(x)=f_0 (x-1)$ if $x\in [0,2)$.
\item 
$g_1(x)=f_1(x)=1-x$ if $x\in [0,1]$ and $g_1(x)=3-x$ if $x\in [0,2)$,
\item
$g_2(x)=x+1 \mod 2$ (or any appropriate map preserving $\{0,1\}$ and interchanging the interior of the intervals $(0,1)$ and $(1,2)$). 
\end{itemize}
These maps are depicted in Figure~\ref{fig.circle}. 
In this case, $\Lambda_{\mathrm{ex}}=\Sigma_3\times \{0,1\}$ and $\Lambda_{\mathrm{core}}=\Sigma_3\times ((0,1)\cup (1,2))$. 
We observe that the IFS $\{g_0,g_1,g_2\}$ does not satisfy the axioms stated in~\cite{DiaGelRam:17a} which would prevent the existence of exposed pieces of dynamics. Although the Axioms Transitivity and CEC (controlled expanding forward/backward covering) can be verified, the Axiom Accessibility is not satisfied (the points $\{0,1,2\}$ cannot ``be reached from outside").

Note that the skew product on $\Sigma_2\times\bS^1$ generated by the fiber maps $\{g_0,g_1\}$ as above is not transitive and has two open ``transitive'' components $\Lambda_{\rm core}^-$ and $\Lambda_{\rm core}^+$ contained in $\Sigma_2\times (0,1)$ and $\Sigma_2\times(1,2)$, respectively, which are glued at the ``exposed" piece $\Sigma_2\times\{0,1\}$. The additional map $g_2$ in the previous example just mixes the two components $\Lambda_{\rm core}^\pm$ while preserving the exposed piece.
\end{remark}

\begin{figure}[h]
\begin{overpic}[scale=.35,bb=0 0 330 330]{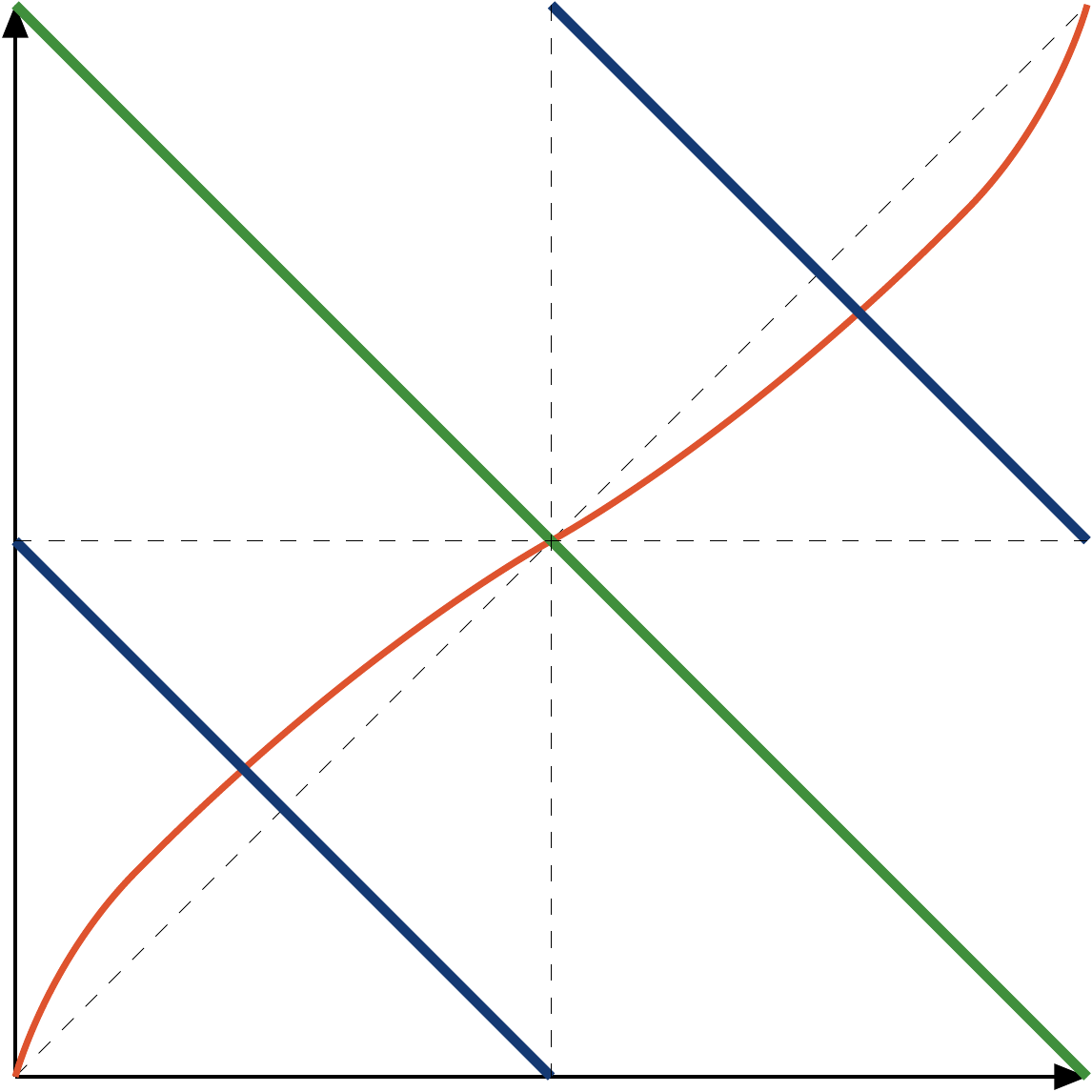}
 \put(65,90){$g_1$}
  \put(15,42){$g_1$}
 \put(80,28){$g_2$}
  \put(56,62.5){$g_0$}
   \put(-1,-8){$0$}
   \put(49,-8){$1$}
 \put(98,-8){$2$}
\end{overpic}
\caption{Fiber maps of the example in Remark~\ref{rem:crcle}}
 \label{fig.circle}
 \end{figure}
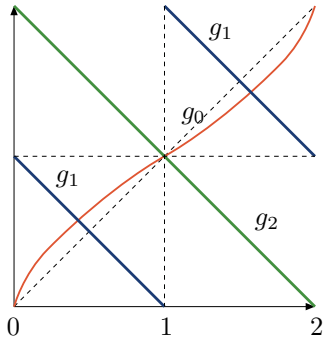

Let $\cM$  be the space of all $F$-invariant  measures and equip it with the weak$\ast$ topology. It is well known that it is a compact metrizable topological space~\cite[Chapter 6.1]{Wal:82}. Denote by $\cM_{\rm erg}=\cM_{\rm erg}(\Sigma_2\times[0,1])$ the subset of ergodic measures. 
We denote by  $\cM_{\rm erg}(\Lambda_{\rm ex})$ the ergodic measures supported on $\Lambda_{\rm ex}$ and by $\cM_{\rm erg}(\Lambda_{\rm core})$ the ergodic measures supported on $\Lambda_{\rm core}$. Observe that
\[
	\cM_{\rm erg}
	=\cM_{\rm erg}(\Lambda_{\rm core})\cup \cM_{\rm erg}(\Lambda_{\rm ex}).
\]	 
We will study  this system by separately looking at measures supported on these two sets. A crucial point for us is how these two components ``glue''.

Given $X=(\xi,x)\in \Sigma_k\times [0,1]$, we consider the \emph{(fiber) Lyapunov exponent} of the map $F$ at $X$ which  is  defined by
\[
	\chi(X)
	\eqdef 
	 \lim_{n\to\pm\infty}\frac{1}{ n}\log\,\lvert (f^n_\xi)'(x)\lvert,
	 \quad\text{ where }\quad
	 f_\xi^n\eqdef f_{\xi_{n-1}}\circ\ldots\circ f_{\xi_0},
\]
where we assume that both limits exist and are equal. Note that  it is nothing but the Birkhoff average of a continuous function. For every $F$-ergodic Borel probability measure $\mu$ the Lyapunov exponent is almost everywhere well defined and constant. This common value of exponents will be called the \emph{Lyapunov exponent} of  $\mu$ and denoted by $\chi(\mu)$. An ergodic measure $\mu$ is  {\emph{nonhyperbolic}} if $\chi(\mu)=0$ and {\emph{hyperbolic}} otherwise. 

Accordingly, we split the set of all \emph{ergodic} measures in $\Lambda_{\rm core}$ and  consider the decomposition
\[
	\cM_{\rm erg}(\Lambda_{\rm core})
	=\cM_{\rm erg,<0}(\Lambda_{\rm core})\cup\cM_{\rm erg,0}(\Lambda_{\rm core})
		\cup\cM_{\rm erg,>0}(\Lambda_{\rm core})
\]
into measures with negative, zero, and positive fiber Lyapunov exponent, respectively. Analogously, we consider
\[	\cM_{\rm erg}(\Lambda_{\rm ex})
	=\cM_{\rm erg,<0}(\Lambda_{\rm ex})\cup\cM_{\rm erg,0}(\Lambda_{\rm ex})
		\cup\cM_{\rm erg,>0}(\Lambda_{\rm ex}).
\]

Properties of the space of measures are summarized in the next theorem. Given  $\cN\subset\cM$,  its \emph{closed convex hull} is the smallest closed convex set containing $\cN$.

\begin{theorem}\label{thepro:measurespace}
Assume that $F$ defined in~\eqref{eq:defF} satisfies the hypotheses (H1) and (H2). Then the space $\cM(\Sigma_2\times[0,1])$ has the following properties:
\begin{enumerate}
\item Periodic orbit measures are dense in the closed convex hull of $\cM_{\rm erg}(\Lambda_{\rm ex})$. 
\item 	Every hyperbolic measure $\cM(\Lambda_{\rm ex})$ has positive weak$\ast$ distance from $\cM(\Lambda_{\rm core})$. 
\item		Every nonhyperbolic measure $\cM(\Lambda_{\rm ex})$ 
can be weak$\ast$ approximated by periodic measures in $\cM_{\rm erg}(\Lambda_{\rm core})$.
\item Each of the components $\cM_{\rm erg,\star}(\Lambda_{\rm ex})$, $\star\in\{<0,0,>0\}$ is nonempty. 
\end{enumerate}
Moreover, if hypotheses (H2'), (H3), and (H4) additionally hold, then
\begin{enumerate}
\item[5.] The set $\cM_{\rm erg,\star}(\Lambda_{\rm core})$, $\star\in\{<0,>0\}$, is nonempty. 
\item[6.] The set $\cM_{\rm erg,<0}(\Lambda_{\rm core})$ and the set $\cM_{\rm erg,>0}(\Lambda_{\rm core})$ are arcwise connected, respectively.
\end{enumerate}
\end{theorem}

The fact that there are ergodic measures with zero Lyapunov exponent and positive entropy in $\cM_{\rm erg}(\Lambda_{\rm core})$ can be shown using methods in \cite{BocBonDia:16}, we refrain from discussing this here. We also refrain from studying how such measures are approached by hyperbolic ergodic measures in $\cM_{\rm erg}(\Lambda_{\rm core})$ as this is much more elaborate and will be part of an ongoing project (see \cite{DiaGelRam:17a} for techniques in a slightly different but technically simpler context).

\begin{figure}[h] \begin{overpic}[scale=.35,bb=0 0 330 330]{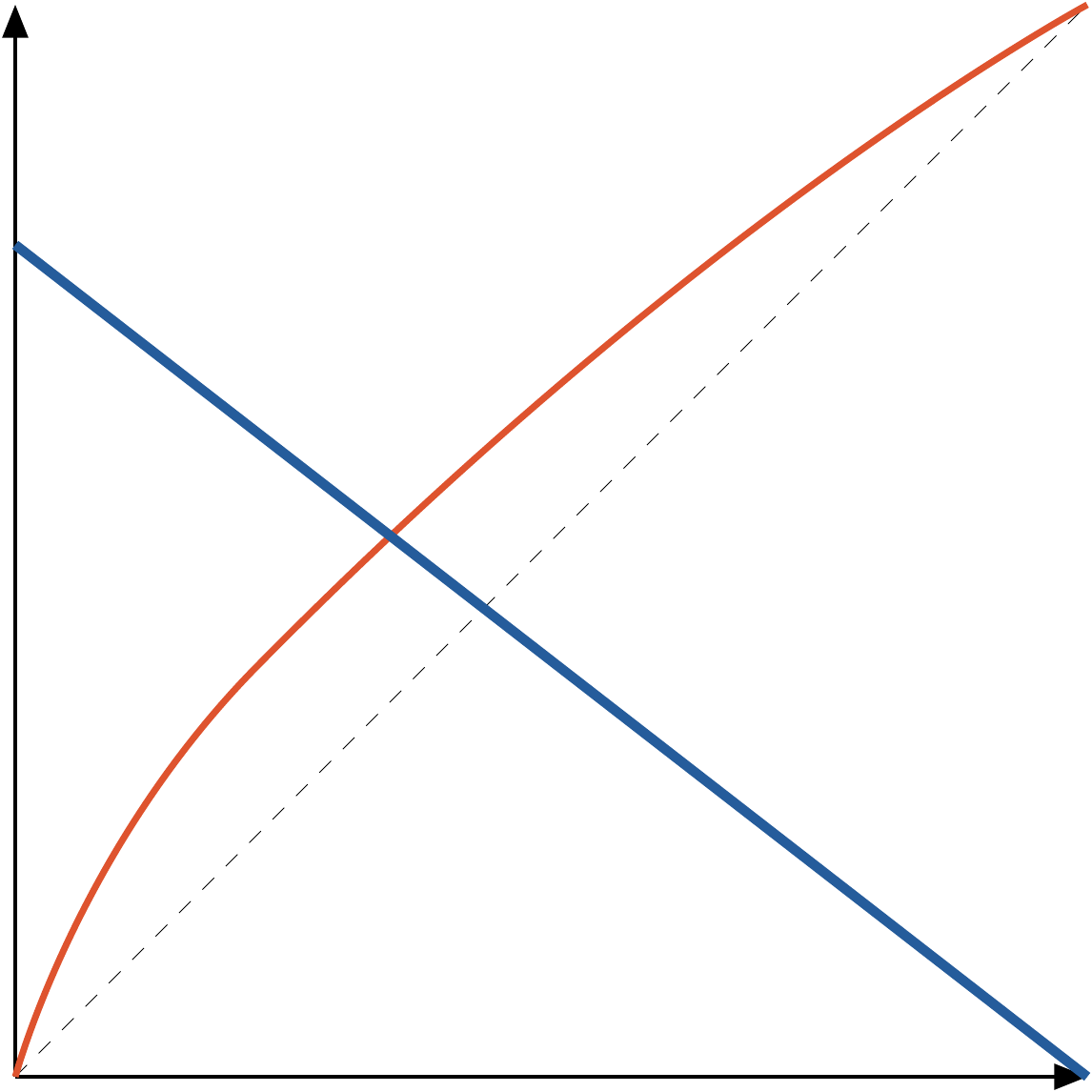}
 \put(60,85){$f_0$}
 \put(80,28){$f_{1,t}(x)=t(1-x)$}
 \put(-8,73){$t$}
 \put(0,-10){$0$}
 \put(96,-10){$1$}
\end{overpic}
\caption{Porcupine-like horseshoes.}
\label{Fig:plike}
\end{figure}
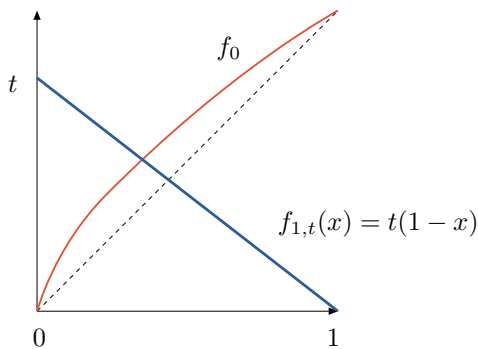

\begin{remark}[Porcupine vs. totally spiny porcupine]
Let us compare the porcupine-like horseshoes corresponding to the  interval maps in Figure~\ref{Fig:plike} with the  ``totally spiny porcupine'' discussed here (corresponding to Figure~\ref{Fig:1}).
Porcupine-like horseshoes were introduced in \cite{DiaHorRioSam:09} as model for internal heterodimensional cycles in horseshoes.
Later these horseshoes were generalized and studied in a series of papers from various points of view: topological (\cite{DiaGel:12,DiaGelRam:13, DiaGelRam:14}, thermodynamical (\cite{LepOliRio:11, DiaGelRam:14,RamSiq:17,RioSiq:18}) and fractal (\cite{DiaMar:15})%
\footnote{The term ``porcupine'' coined in \cite{DiaGel:12} refers to the rich topological fiber structure of the homoclinic class, which is simultaneously composed of  uncountable many fibers which are continua and  uncountable many ones which are just points. In this paper, all fibers are full intervals.}. 
This line of research is also closely related to the study of so-called \emph{bony attractors} and \emph{sets} (see~\cite{IlyShi:17} for a survey and references). 
One important motivation to study those models is that they serve as a prototype of partially hyperbolic dynamics.

Let us consider the map $F_t$ defined as in \eqref{eq:defF} but with the maps $f_0,f_{1,t}$ as in Figure~\ref{Fig:plike} in the place of $f_0,f_1$ in Figure~\ref{Fig:1}. Let $\Gamma^t$ be the maximal invariant set of $F_t$. In the above cited porcupine-like horseshoes, one also splits the maximal invariant set $\Gamma^t$ (which is nonhyperbolic and transitive) into two parts $\Gamma^t_{\rm ex}$ and $\Gamma^t_{\rm core}$ in the same spirit as in~\eqref{eq:defLambda} (and with analogous properties as in Proposition~\ref{pro:22}). In that case $\Gamma^t_{\rm ex}$ consists only of one fiber expanding point $Q=(0^\bZ,0)$ and $\Gamma^t_{\rm core}$ is its complement that contains the fiber contracting point $P=(0^\bZ,1)$. The space of ergodic measures of $\Gamma^t$ splits into \emph{two}  components, each of them connected but at positive distance from each other, which are $\{\delta_Q\}$ and $\cM_{\rm erg}(\Gamma^t_{\rm core})$ (see, in particular,~\cite{LepOliRio:11}). In the transition from a porcupine to a totally spiny porcupine (which occurs at $t=1$), the space of ergodic measures becomes connected (stated in Theorem~\ref{thepro:measurespace}) and this happens as follows. The measures $\delta_Q$ and $\delta_P$ form part of the space of ergodic measures of an abstract horseshoe $\Lambda_{\rm ex}$. At the same time, the measure $\delta_P$ detaches from $\cM_{\rm erg}(\Lambda_{\rm core})$ which is a consequence of the fact that the saddle $P$ is not homoclinically related to any saddle in $\Lambda_{\rm core}$, similarly for $Q$. The components $\cM_{\rm erg}(\Lambda_{\rm ex})$ and $\cM_{\rm erg}(\Lambda_{\rm core})$ become glued through nonhyperbolic measures.
\end{remark}
 
 \begin{theorem}\label{the:1}
 Assume that $F$ defined in~\eqref{eq:defF} satisfies the hypotheses (H1) and (H2). 
 	Then there is a unique measure $\mu_{\rm max}^{\rm ex}$ of maximal entropy $\log2$ in $\cM_{\rm erg}(\Lambda_{\rm ex})$ and its Lyapunov exponent is given by
\[
	\chi(\mu_{\rm max}^{\rm ex})
	= \frac14\left(\log\, f_0'(0)+\log\, f_0'(1)
						+\log\, \lvert f_1'(0)\rvert+\log\, \lvert f_1'(1)\rvert\right).
\]
	Moreover, if the measure $\mu_{\rm max}^{\rm ex}$ is hyperbolic then there exists at least one measure of maximal entropy in $\cM_{\rm erg}(\Lambda_{\rm core})$. 
	More precisely, if the measure $\mu_{\rm max}^{\rm ex}$ has positive (negative) Lyapunov exponent then there exists a measure of maximal entropy with nonpositive (nonnegative) exponent  in $\cM_{\rm erg}(\Lambda_{\rm core})$.
 \end{theorem} 

 Note that the topological structure of $\cM(\Lambda_{\rm ex})$ (items 1. and 4.  in Theorem~\ref{the:1}) are immediate consequences of the fact that the dynamics of $F$ on $\Lambda_{\rm ex}$ is conjugate to a subshift of finite type (see Section~\ref{sec:twinsetc} for details).  

 Note that the under the hypotheses of the above theorem, we do not know if the measure of maximal entropy in $\Lambda_{\rm core}$ is hyperbolic or not.

\begin{remark}
	In view of Theorem~\ref{the:1}, choosing the derivatives of the fiber maps at $0$ and $1$ appropriately, one obtains one measure of maximal entropy $\mu_{\rm max}^{\rm ex}$ which is nonhyperbolic. Note that condition (H4) is  incompatible with such a choice, and hence it is unclear if the system is transitive (compare Proposition \ref{pro:22}).
	
	Similar arguments apply to the examples discussed in Remark~\ref{rem:crcle}.
	It is interesting to compare to the results in \cite{TahYan:} where maps with ``sufficiently high entropy measures" are always hyperbolic, though there a key ingredient is accessibility which is missing here.
\end{remark}

	In Appendix \ref{a.parabolic}, we provide examples where the system is transitive and exhibits a nonhyperbolic measure of maximal entropy in $\cM_{\rm erg}(\Lambda_{\rm ex})$, proving the following theorem.
	
\begin{theorem}\label{theoremA8a}
	There are maps $\tilde F$ defined as in~\eqref{eq:defF} whose fiber maps $\tilde f_0,\tilde f_1$ satisfy\begin{enumerate}
\item $\tilde f_0$ has (exactly) two fixed points $\tilde f_0(0)=0$ and $\tilde f_0(1)=1$,
with  $\tilde f_0'(0)=1=\tilde f_0'(1)$,
\\[-0.4cm]
\item 
$\tilde f_1(x)=1-x$, 
\end{enumerate}
such that $\tilde F$ is topologically transitive and that every pair of fiber expanding hyperbolic periodic orbits and every pair of fiber contracting hyperbolic periodic orbits in $\Lambda_{\rm core}$ are homoclinically related, respectively.  In particular, the unique measure of maximal entropy in $\cM_{\rm erg}(\Lambda_{\rm ex})$ is nonhyperbolic. 
\end{theorem}

Note that in the above theorem this measure is also a measure of maximal entropy in $\cM_{\rm erg}(\Lambda)$, however we do not know if there is some hyperbolic measure  of maximal entropy  in $\cM_{\rm erg}(\Lambda)$. 

\emph{Mutatis mutandi}, we can perform a version of the map $\tilde F$ in $\Sigma_3\times \mathbb{S}^1$ as in Remark~\ref{rem:crcle}.
 
 Finally, in Appendix~\ref{app:bad} we present an example with a unique measure of maximal entropy which is nonhyperbolic and supported on $\Lambda_{\rm ex}$. However this example fails to be transitive. 

\begin{theorem}\label{theoremA8b}
		There are maps $F$ defined as in~\eqref{eq:defF} whose fiber maps $ f_0, f_1$ satisfy\begin{enumerate}
\item $ f_0$ has (exactly) two fixed points $ f_0(0)=0$ and $ f_0(1)=1$,
with  $ f_0'(0)=1=f_0'(1)$,
\\[-0.4cm]
\item 
$f_1(x)=1-x$. 
\end{enumerate}
such that $ F$ is not topologically transitive and has a unique measure of maximal entropy supported on $\Lambda_{\rm ex}$, which is nonhyperbolic.
\end{theorem}

One of the key properties of the class of examples in the above theorem is that $f_0$ is conjugate to its inverse $f_0^{-1}$ by $f_1$. The proof of the result is based on an analysis of random walks on $\bR$ and of somewhat different flavor. 

%------------------------------------------------------------------------------------------------------
\section{Symmetric, mirror, and twin measures}\label{sec:twinsetc}
%------------------------------------------------------------------------------------------------------

Recalling well-known facts about shift spaces, we will see that there is a unique measure of maximal entropy for $F|_{\Lambda_{\rm ex}}$ and we will deduce that, in the case this measure is hyperbolic, there is (at least) one ``twin'' measure in $\Lambda_{\rm core}$ with the same (maximal) entropy. The latter is either hyperbolic with opposite sign of its exponent or nonhyperbolic.

Recall that on the full shift $\sigma\colon\Sigma_2\to\Sigma_2$ there is a unique measure $\widehat\nu_{\rm max}$ of maximal entropy $\log2$ which is the $(\frac12,\frac12)$-Bernoulli measure. 

To study the structure of  the invariant set $\Lambda_{\rm ex}$, consider the ``first level'' rectangles  $\bC_k\eqdef \{\xi\in\Sigma_2\colon \xi_0=k\}$ and the subsets
\[
	 \widehat\bC_{0_L}\eqdef \bC_0\times\{0\} ,\quad
	 \widehat\bC_{1_L}\eqdef \bC_1\times\{0\} ,\quad
	 \widehat\bC_{0_R}\eqdef \bC_0\times\{1\} ,\quad
	 \widehat\bC_{1_R}\eqdef \bC_1\times\{1\} ,
\]
of $\Sigma_2 \times [0,1]$.
Consider the transition matrix $A$  given by
\[
	A\eqdef
	\left(\begin{matrix}
	1&1&0&0\\0&0&1&1\\0&0&1&1\\1&1&0&0
	\end{matrix}\right).
\]
This matrix codes the transitions between the symbols $\{0_L,1_L,0_R,1_R\}$ modelling the transitions
between the sets   $\widehat\bC_{0_L},  \widehat\bC_{1_L}, \widehat\bC_{0_R},  \widehat\bC_{1_R}$
by the map $F$.
More precisely, note that  the restriction of $F|_{\Lambda_{\rm ex}}$ is topologically conjugate to the subshift of finite type $\sigma_A\colon\Sigma_A\to\Sigma_A$ by means of a  map $\varpi  \colon
\Lambda_{\mathrm{ex}}\to \Sigma_A$. Note that there is a unique measure of maximal entropy $\nu^{\rm ex}_{\rm max}$ for $\sigma_A\colon\Sigma_A\to\Sigma_A$%
\footnote{Note that this measure is the Parry measure associated to the topological Markov chain $\sigma_A$, see~\cite[Theorem 8.10]{Wal:82}.}. 
Note that $h_{\nu^{\rm ex}_{\rm max}}(\sigma)=\log2$. 
Hence, by conjugation, the measure $\mu^{\rm ex}_{\rm max}=(\varpi^{-1})_\ast\nu^{\rm ex}_{\rm max}$ is the unique measure of maximal entropy $\log2$ for $F\colon\Lambda_{\rm ex}\to\Lambda_{\rm ex}$.

We define the following projection 
\[\begin{split}
	\Pi\colon\Sigma_A\to\Sigma_2,
	\quad
	&\Pi (\ldots i_{-1}.i_0\ldots)\eqdef(\ldots \xi_{-1}.\xi_0\ldots),\\
	&\xi_k\eqdef\begin{cases}
		0&\text{ if }i_k\in\{0_L,0_R\}\\
		1&\text{ if }i_k\in\{1_L,1_R\}.
	\end{cases}
\end{split}\]	 
It is immediate to check that
\[
	\widehat\nu_{\rm max}
	= \Pi_\ast\nu^{\rm ex}_{\rm max}.
\]

We say that the symbols $i_R$, $j_L$ are the \emph{mirrors} of $i_L$ and  $j_R$, respectively, for $i$ and $j$ in $\{0,1\}$ and denote $i_R=\bar{i_L}$ and $j_L=\bar{j_R}$. Given a sequence $\xi=(\ldots \xi_{-1}.\xi_0\ldots)\in\Sigma_A$, we define by $\bar{\xi}=(\ldots \bar{\xi_{-1}}.\bar{\xi_0}\ldots)$ the \emph{mirrored} sequence of $\xi$. Note that $\bar\xi\in\Sigma_A$. Given a subset $B\subset\Sigma_A$, we denote by $\bar B\eqdef \{\bar\xi\colon\xi\in B\}$ its \emph{mirrored} set.
 
 Now we are ready to define \emph{symmetric sets} and \emph{measures}.

\begin{definition}[Symmetric sets and measures]
	A measurable set $B\subset\Sigma_A$ is \emph{symmetric} if $B=\bar B$. We say that $B$ is \emph{symmetric $\nu$-almost surely} if $\nu(B\Delta\bar B)=0$. A measure $\nu\in\cM(\Sigma_A)$ is \emph{symmetric} if $\nu(\bar{B})=\nu(B)$ for every $B\subset\Sigma_A$. If a measure is not symmetric then we call it \emph{asymmetric}. A measure $\mu\in\cM(\Lambda_{\mathrm{ex}})$ is \emph{symmetric} if $\varpi_\ast \mu$ is symmetric, otherwise we call it \emph{asymmetric}.
	
	We denote by $\cM_{\rm erg}^{\rm sym}(\Lambda_{\rm ex})$ and $\cM_{\rm erg}^{\rm asym}(\Lambda_{\rm ex})$ the sets of symmetric and asymmetric ergodic measures in $\Lambda_{\rm ex}$, respectively.
 \end{definition}

We will use the following lemma.
\begin{lemma}\label{lem:lemma}
Let $\nu\in\cM(\Sigma_A)$ be a symmetric measure. Then any set $B\subset\Sigma_A$ which is $\nu$-almost symmetric satisfies
\[
	\nu((\Pi^{-1}\circ\Pi)(B) )
	= \nu(B). 
\]
\end{lemma}

\begin{proof}
Indeed, by $\nu$-almost symmetry of $B$, setting $C\eqdef B\cap\bar B$, $D\eqdef \bar B\setminus C$, and $E\eqdef B\setminus C$, we have $B=E\cup C$,
$\nu(C)=\nu(\bar C)=\nu(B)=\nu(\bar B)$, and $\nu(D)=\nu(E)=0$. Hence $\nu(\bar D)=\nu(D)=0$. Observing that 
\[
	(\Pi^{-1}\circ \Pi)(B) 
	=B\cup\bar B = E\cup C\cup D
\]
 the claim follows.
\end{proof}

\begin{lemma}\label{l.mirrormeasure}
	For every $\nu\in\cM_{\rm erg}(\Sigma_A)$, there exist at most one measure $\bar\nu\in\cM_{\rm erg}(\Sigma_A)$, $\bar\nu\ne\nu$, such that $\Pi_\ast\bar\nu=\Pi_\ast\nu$. There is no such measure if, and only if, $\nu$ is symmetric. 
\end{lemma}

\begin{proof}%[Proof of Lemma~\ref{l.mirrormeasure}] 
It suffices to observe that the product $\sigma$-algebra of Borel measurable sets of $\Sigma_A$ is generated by the semi-algebra generated by the family of all finite cylinder sets $\{[i_k\ldots i_\ell]\}$. Note also that the mirror $\bar C$ of a cylinder $C$ in $\Sigma_A$ is again a cylinder in $\Sigma_A$. 
Now given $\nu\in\cM_{\rm erg}(\Sigma_A)$, define a measure $\bar\nu$ by setting $\bar\nu(C)\eqdef\nu(\bar C)$ for every cylinder $C$ and extend it to the generated $\sigma$-algebra.

By definition, we immediately obtain that $\Pi_\ast\bar\nu=\Pi_\ast\nu$ and that $h_{\bar\nu}(\sigma_A)=h_\nu(\sigma_A)$. 

To prove that $\nu$ and $\bar\nu$ are the only ergodic measures satisfying 
$\Pi_\ast\bar\nu=\Pi_\ast\nu$, by contradiction assume that there exists $\widehat\nu\in\cM_{\rm erg}(\Sigma_A)$, $\bar\nu\ne\widehat\nu\ne\nu$  satisfying $\Pi_\ast\widehat\nu=\Pi_\ast\nu$. Consider the measure $\widetilde\nu\eqdef\frac12(\nu+\bar\nu)$. Note that $\widetilde\nu$ is symmetric. Also note that $\Pi_\ast\widetilde\nu=\Pi_\ast\widehat\nu$. Finally note that $\widetilde\nu$ is singular with respect to $\widehat\nu$ and hence there is a set $B\subset\Sigma_A$ satisfying $\widetilde\nu(B)=0=\widehat\nu(B^c)$. 
Since $\widetilde\nu$ is symmetric, we have $\widetilde\nu(\bar B)=\widetilde\nu(B)$. So we obtain $\widetilde\nu(\bar B\triangle B)=0$ and hence $B$ is $\widetilde\nu$-almost symmetric. Hence,  we have
\[\begin{split}
	0 
	&= \widetilde\nu(B) \\
	(\text{by Lemma~\ref{lem:lemma} })&= \widetilde\nu((\Pi^{-1}\circ\Pi)(B))
	= \Pi_*\widetilde\nu(\Pi(B)) \\
	(\text{since }\Pi_\ast\widetilde\nu=\Pi_\ast\widehat\nu\,)
	&= \Pi_*\widehat\nu(\Pi(B)) 
	= \widehat\nu((\Pi^{-1}\circ\Pi)(B)) \geq \widehat\nu(B) =1,
\end{split}\]
a contradiction. This proves that $\bar\nu$ is uniquely defined.

By definition, $\nu$ is symmetric if, and only if, $\bar\nu=\nu$.
\end{proof}

\begin{definition}[Mirror measure]
We call the measure $\bar\nu$ provided by Lemma~\ref{l.mirrormeasure} the \emph{mirror measure} of $\nu$. 
We call the measure $\varpi^{-1}_\ast\bar\nu\in\cM(\Lambda_{\mathrm{ex}})$ the \emph{mirror measure} of $\mu=\varpi^{-1}_\ast \nu$ and denote it by $\bar\mu$.
\end{definition}

The following is an immediate consequence of Lemma~\ref{l.mirrormeasure} and the uniqueness of the measure of maximal entropy.

\begin{corollary}\label{cor:maxentsym}
	The measure of maximal entropy $\mu^{\rm ex}_{\rm max}$ is symmetric.
\end{corollary}

\begin{lemma}\label{lem:outro}
	If $\bar\mu$ is a mirror measure of $\mu\in\cM_{\rm erg}(\Lambda_{\rm ex})$ then $h_{\bar\mu}(F)=h_\mu(F)$. Moreover, we have 
\[
	\chi(\bar\mu)+\chi(\mu)
	=N(0)\log ( f_0'(0)\cdot f_0'(1)) +N(1)\log( f_1'(0)\cdot  f_1'(1)),
\]
where $N(0)\eqdef\mu(\Sigma_2\times\{0\})$ and $N(1)\eqdef\mu(\Sigma_2\times\{1\})$. 
\end{lemma}

\begin{proof}
%Note again that $\pi=\Pi\circ\varpi$. Note that $h_{\pi_\ast\mu}(\sigma)\le h_{\bar\mu}(F)$ and $h_{\pi_\ast\mu}(\sigma)\le h_\mu(F)$. On  the other hand, by~\cite{LedWal:77}
%\[
%	\max\{h_\mu(F),h_{\bar\mu}(F)\}
%	\le \sup_{m\colon\pi_\ast m=\pi_\ast\mu}h_m(F)
%	= h_{\pi_\ast\mu}(\sigma)+\int h_{\rm top}(F,\pi^{-1}(\xi))\,d\pi_\ast\mu(\xi).
%\]
%Since $\pi$ is 2-1, we have $h_{\rm top}(F,\pi^{-1}(\xi))=0$ for every $\xi$.
%Thus, we conclude $h_\mu(F)=h_{\bar\mu}(F)=h_{\pi_\ast\mu}(\sigma)$. 
%
Let $\nu=\varpi_\ast\mu$. 
%To verify now the property of the Lyapunov exponents, i
It suffices to observe that a sequence $\xi$ is $\nu$-generic if, and only if, $\bar\xi$ is $\bar\nu$-generic and to do the straightforward calculation.
\end{proof}

\begin{definition}
Given an ergodic measure $\mu\in\cM_{\rm erg}(\Sigma_2\times[0,1])$, an ergodic measure $\widetilde\mu\in\cM_{\rm erg}(\Sigma_2\times[0,1])$, $\widetilde\mu\ne\mu$, is called a \emph{twin measure} of $\mu$ if $\pi_\ast\widetilde\mu=\pi_\ast\mu$. 
 \end{definition}

%\begin{lemma}\label{lem:notbound}
%	Every $\mu\in\cM_{\rm erg}(\Lambda_{\rm ex})$ has a twin measure if and only if $\mu$ is asymmetric. 
Note that the above immediately implies that if $\mu\in\cM_{\rm erg}(\Lambda_{\rm ex})$ is symmetric then all its twin measures are in $\cM_{\rm erg}(\Lambda_{\rm core})$.
%\end{lemma}
 
%\begin{proof}
%By conjugation $\varpi$ between $F|_{\Lambda_{\rm ex}}$ and $\sigma_A|_{\Sigma_A}$ and by Lemma \ref{l.mirrormeasure}, there can be at most one other ergodic measure in $\cM_{\rm erg}(\Lambda_{\rm ex})$ which project to the same measure on $\Sigma_2$, namely $\varpi^{-1}_\ast\bar\nu$, where $\nu=\varpi^{-1}_\ast\mu$ and $\bar\nu$ is the mirror measure of $\nu$. \textcolor{blue}{For symmetric $\mu$, no such mirror and hence twin exists.}
%
%On the other hand, for asymmetric $\mu$, this mirror measure provides a twin measure.
%\end{proof}

\begin{lemma}[Existence of twin measures]\label{lem:simmea}
	For every measure $\lambda\in\cM(\Sigma_2)$ there exist a measure $\mu_1\in\cM(\Sigma_2\times[0,1])$ satisfying $\pi_\ast\mu_1=\lambda$ and $\chi(\mu_1)\ge0$ and a measure $\mu_2\in\cM(\Sigma_2\times[0,1])$ satisfying $\pi_\ast\mu_2=\lambda$ and $\chi(\mu_2)\le0$.

	Moreover, if $\lambda$ was ergodic then $\mu_1$ and $\mu_2$ can be chosen ergodic.
\end{lemma}

Note that the measures $\mu_1$ and $\mu_2$ in the above lemma may coincide.

\begin{proof}
First observe that $\lambda\in\cM(\Sigma_2)$ is weak$\ast$ approximated by measures $\lambda_\ell\in\cM(\Sigma_2)$ supported on periodic sequences. 

For each such measure $\lambda_\ell$ there exists a measure $\mu_\ell\in\cM(\Sigma_2\times[0,1])$ which is supported on a $F$-periodic orbit in $\Sigma_2\times[0,1]$ and satisfies $\pi_\ast\mu_\ell=\lambda_\ell$ and $\chi(\mu_\ell)\ge0$. 
Indeed, assume that $\lambda_\ell$ is supported on the orbit of a periodic sequence $\xi\in\Sigma_2$ of period $n$. Recall that the fiber maps $f_0$ and $f_1$ and hence the map $f_\xi^n$ preserve the boundary $\{0,1\}$. Hence, this map $f_\xi^n$ has a fixed point $x\in[0,1]$ satisfying $\lvert (f_\xi^n)'(x)\rvert\ge1$. Now observe that the orbit of $(\xi,x)$ is $F$-periodic of period $n$ and taking the measure $\mu_\ell$ supported on it we have $\frac1n\log\,\lvert (f_\xi^n)'(x)\rvert=\chi(\mu_\ell)$.

Now take any weak$\ast$ accumulation point $\mu$ of the sequence $(\mu_\ell)_\ell$. Note that by continuity of $\pi_\ast$ we have $\pi_\ast\mu=\lambda$. 

If $\lambda$ was ergodic, $\mu$ might not be ergodic. However, any ergodic measure in the ergodic decomposition of $\mu$ also projects to $\lambda$ and hence there must exist one measure $\mu'$ in this decomposition satisfying  $\chi(\mu')\ge0$. 

The same arguments work for the case $\chi(\cdot)\le0$.
\end{proof}

\begin{corollary}
	For every hyperbolic symmetric ergodic measure $\mu\in\cM_{\rm erg}(\Lambda_{\rm ex})$ there exists an ergodic twin measure  $\widetilde\mu\in\cM_{\rm erg}(\Sigma_2\times(0,1))$, $\widetilde\mu\ne\mu$, satisfying $h_{\widetilde\mu}(F)=h_\mu(F)$.
\end{corollary}

\begin{proof}
	Assume that $\chi(\mu)>0$, the other case $\chi(\mu)<0$ is analogous. Let $\lambda\eqdef\pi_\ast\mu$.  By Lemma~\ref{lem:simmea}, there exists a twin measure $\widetilde\mu\in\cM_{\rm erg}(\Sigma_2\times[0,1])$ of $\mu$ satisfying $\chi(\widetilde\mu)\le0$. 
%	\margem{K: added here the entropy argument and changed the end}
%By Lemma~\ref{lem:outro}, we have $h_{\widetilde\mu}(F)=h_\mu(F)$.
Note that $h_{\pi_\ast\widetilde\mu}(\sigma)\le h_{\widetilde\mu}(F)$ and $h_{\pi_\ast\mu}(\sigma)\le h_\mu(F)$. On  the other hand, by~\cite{LedWal:77}
\[
	\max\{h_\mu(F),h_{\widetilde\mu}(F)\}
	\le \sup_{m\colon\pi_\ast m=\pi_\ast\mu}h_m(F)
	= h_{\pi_\ast\mu}(\sigma)+\int h_{\rm top}(F,\pi^{-1}(\xi))\,d\pi_\ast\mu(\xi).
\]
Since $\pi$ is 2-1, we have $h_{\rm top}(F,\pi^{-1}(\xi))=0$ for every $\xi$.
Thus, we conclude $h_\mu(F)=h_{\widetilde\mu}(F)=h_{\pi_\ast\mu}(\sigma)$. 

By conjugation $\varpi$ between $F|_{\Lambda_{\rm ex}}$ and $\sigma_A|_{\Sigma_A}$, there can be at most one other ergodic measure in $\cM_{\rm erg}(\Lambda_{\rm ex})$ which project to the same measure on $\Sigma_2$, namely $\varpi^{-1}_\ast\bar\nu$, where $\nu=\varpi^{-1}_\ast\mu$ and $\bar\nu$ is the mirror measure of $\nu$. For symmetric $\mu$, no such mirror exists. Hence, we must have $\widetilde\mu\in\cM_{\rm erg}(\Sigma_2\times(0,1))$.
%By Lemma~\ref{lem:notbound}, using symmetry and ergodicity, we must have $\widetilde\mu\not\in\cM_{\rm erg}(\Lambda_{\rm ex})$. Hence, $\widetilde\mu\in\cM_{\rm erg}(\Sigma_2\times(0,1))$.
\end{proof}

By Corollary~\ref{cor:maxentsym}, the above applies in particular to $\mu_{\rm max}^{\rm ex}$.

\begin{corollary}\label{cor:other}
		If the measure of maximal entropy $\mu_{\rm max}^{\rm ex}\in\cM_{\rm erg}(\Lambda_{\rm ex})$ is hyperbolic then there exists an ergodic twin measure of maximal entropy $\widetilde\mu\in\cM_{\rm erg}(\Lambda_{\rm core})$ such that $\chi(\widetilde\mu)\chi(\mu_{\rm max}^{\rm ex})\le0$.
\end{corollary}

\begin{proof}[Proof of Theorem~\ref{the:1}]
As recalled already, there is a unique measure of maximal entropy for $F|_{\Lambda_{\rm ex}}$ and its Lyapunov exponents can be easily calculated.
The fact that there may exist another  measure of maximal entropy for $F|_{\Lambda_{\rm core}}$  follows immediately from Corollary~\ref{cor:other}.
\end{proof}

 %------------------------------------------------------------------------------------------------------
\section{Approximations of boundary measures}\label{sec:measures}
%------------------------------------------------------------------------------------------------------

This section discusses the approximation of measures in $\cM(\Lambda_{\rm ex})$ by (ergodic) measures in $\cM(\Lambda_{\rm core})$. In particular, we will complete the proof of Theorem~\ref{thepro:measurespace}. We will always work with the system satisfying hypotheses (H1) and (H2).

Recall again that $\cM$ equipped with the weak$\ast$ topology it is a compact metrizable topological space~\cite[Chapter 6.1]{Wal:82}. 
Recall that $X\in\Sigma_2\times[0,1]$ is a \emph{generic point} of a measure $\mu\in\cM(\Sigma_2\times[0,1])$ if the sequence $\frac1n(\delta_X+\delta_{F(X)}+\ldots+\delta_{F^{n-1}(X)})$ converges to $\mu$ in the weak$\ast$ topology, where $\delta_Y$ denotes the Dirac measure supported in $Y$. Recall that for every ergodic measure there exists a set of generic points with full measure.

Given $\delta\in(0,1/2)$, we consider the local distortion map
\begin{equation}\label{eq:defDelta}
	\Delta(\delta)
	\eqdef \max_{i=0,1}
		\left\{\max_{z\in [0,\delta]}\left\lvert\log\frac{\lvert f_i'(z)\rvert}{\lvert f_i'(0)\rvert}\right\rvert,
		\max_{z\in [1-\delta,1]}\left\lvert\log\frac{\lvert f_i'(z)\rvert}{\lvert f_i'(1)\rvert}
		\right\rvert\right\}.
\end{equation}
Note that $\Delta(\delta)\to0$ as $\delta\to0$. 
We state the following simple facts without proof.
 
\begin{lemma}\label{lem:affineapprox}
	For every $\delta\in(0,1/2)$ and every $x\in [0,\delta]$ we have
\[
	e^{-\Delta(\delta)}
	\le \frac{ f_i'(x)}{ f_i'(0)},
		\frac{ f_i'(1-x)}{ f_i'(1)}
	\le e^{\Delta(\delta)}	
\]	
and 
\[
	e^{-\Delta(\delta)}
	\le \frac{\lvert f_i(x)-f_i(0)\rvert}{\lvert x\rvert\lvert f_i'(0)\rvert},
		\frac{\lvert f_i(1-x)-f_i(1)\rvert}{\lvert 1-(1-x)\rvert\lvert f_i'(1)\rvert}
	\le e^{\Delta(\delta)}			.
\]
\end{lemma}

\begin{proposition}\label{pro:measzeroisapprox}
	For every $\mu\in\cM(\Lambda_{\rm ex})$ satisfying $\chi(\mu)=0$ there exists a sequence $(\mu_k)_k\subset\cM_{\rm erg}(\Lambda_{\rm core})$ of measures supported on periodic orbits which converge to $\mu$ in the weak$\ast$ topology.
\end{proposition}

\begin{proof}
	Let $\mu$ be an invariant measure supported in $\Lambda_{\rm ex}$ and satisfying the hypothesis $\chi(\mu)=0$ and let $X=(\xi,x)\in\Lambda_{\rm ex}$ be a $\mu$-generic point. Hence $\chi(X)=0$. Note that  $\xi$ hence has infinitely many symbols $1$ by our hypothesis $f_0'(0)\ne1\ne f_0'(1)$. 
Hence, without loss of generality, we can assume that $x=1$.
		
	Given the sequence $\xi=(\ldots\xi_{-1}.\xi_0\xi_1\ldots)$, for $n\ge1$ define
\[\begin{split}
	p_n
	&\eqdef \card\big\{i\in\{0,\ldots,n-1\}\colon \xi_i=0, \card\{j<i\colon\xi_j=1\}\text{ even}\big\},\\
	q_n
	&\eqdef \card\big\{i\in\{0,\ldots,n-1\}\colon \xi_i=0, \card\{j<i\colon\xi_j=1\}\text{ odd}\big\},\\
	r_n
	&\eqdef \card\big\{i\in\{0,\ldots,n-1\}\colon \xi_i=1, \card\{j<i\colon\xi_j=1\}\text{ even}\big\},\\
	s_n
	&\eqdef \card\big\{i\in\{0,\ldots,n-1\}\colon \xi_i=1, \card\{j<i\colon\xi_j=1\}\text{ odd}\big\}.
\end{split}\]
Note that $n=p_n+q_n+r_n+s_n$. Let
\begin{equation}\label{eq:previous}
	\phi(n)
	\eqdef 
	p_n\log\,\lvert f_0'(1)\rvert +q_n\log\,\lvert f_0'(0)\rvert
	+r_n\log\,\lvert f_1'(1)\rvert + s_n\log\,\lvert f_1'(0)\rvert.
\end{equation}
Observe that $\phi(n)=\log\,\lvert (f_\xi^n)'(1)\rvert$ and hence $\chi(X)=0$ implies
\[
	\lim_{n\to\infty}\frac{\phi(n)}{n}=0.
\]
Let 
\[
	\psi(n)
	\eqdef \max_{i=1,\ldots,n}\lvert \phi(i)\rvert
\]
and note that 
\begin{equation}\label{eq:limpsin}
	\lim_{n\to\infty}\frac{ \psi(n)}{n}=0.
\end{equation}
Let $B=\{0,1\}$. Given a fiber point $x_0\in(0,1)$ let us use the following notation of its orbit under the fiber dynamics determined by the sequence $\xi$:
\begin{equation}\label{def:sequence}
	 x_i
	 \eqdef f_\xi^i(x_0).
\end{equation}

\noindent\textbf{Partially affine case:}
To sketch the idea of the proof, assume for a moment that $f_0|_{I_\delta}$ and $f_1|_{I_\delta}$ are affine, where $I_\delta=[0,\delta]\cup[1-\delta,1]$  for some small $\delta>0$.
 Note that for given $n$ and a point $x_0 \in [1-\delta, 1)$ satisfying 
\begin{equation}\label{eq:condition}
	x_i\in I_\delta
	\quad\text{ for all }\quad
	i\in \{0,\ldots,n-1\}
\end{equation}
we have
\begin{equation}\label{eq:estimprelim}
	\frac{\dist(x_{i+1},B)}{\dist(x_i,B)}
	= e^{\phi(i+1)-\phi(i)},
\end{equation}
where $\dist(x,B)$ denotes the distance of $x$ from a set $B$.
Hence
\[
	 \frac{\dist(x_n,B)}{\dist(x_0,B)}
	 = e^{\phi(n)}.
\] 
This implies
\begin{equation}\label{eq:estim}
	e^{-\psi(n)}
	\le \frac{\dist(x_i,B)}{\dist(x_0,B)}
	\le e^{\psi(n)}
	\quad\text{ for all }\quad
	i\in \{0,\ldots,n\}.
\end{equation}
Note that~\eqref{eq:condition} is satisfied provided $x_0$ was chosen to satisfy $\dist(x_0,B)<\delta e^{-\psi(n)}$. 

Note that $e^{-\psi(n)}$ may not converge to $0$. For this reason, let us choose  
\begin{equation}\label{eq:defdeltan}
	\delta(n)
	\eqdef \delta e^{-2\max\{\psi(n),\sqrt{n}\}}.
\end{equation}
Note that
\begin{equation}\label{eq:deltanasym}
	\lim_{n\to\infty}\delta(n)e^{\psi(n)}=0.
\end{equation}

Let now $n$ be a sufficiently large integer such that  $\card\{j\le n-1\colon\xi_j=1\}$ is odd. Note that this implies $f_\xi^n$ is orientation reversing.
Let $N(n)$ be the smallest positive integer such that $x_0\eqdef f_0^{N(n)}(1/2)\in[1-\delta(n),1)$. Note that 
\begin{equation}\label{e.Nn}
	N(n)
	\sim \lvert\log\delta(n)\rvert,
\end{equation}
where the approximation is up to some universal multiplicative factor, independent on $n$. Note that $f_0^{N(n)}$ is orientation preserving.
We now apply the above arguments to the chosen point $x_0$. We consider the sequence $(x_i)_{i=0}^n$ as defined in~\eqref{def:sequence}. First, note that $\dist(x_0,B)\sim\delta(n)$ and with~\eqref{eq:estim} we have 
\[
	\delta(n) e^{-\psi(n)} \le
	x_i
	\le \dist(x_0,B) e^{\psi(n)}
	\le \delta(n) e^{\psi(n)}
	\quad\text{ for all }\quad
	i\in\{0,\ldots,n\}.
\]	
Further note that $x_n$ by our choice of $n$ is close to $0$.
Then let $M(n)$ be the smallest positive integer such that $f_0^{M(n)}(x_n)\ge1/2$. 
Note that $f_0^{M(n)}$ is orientation preserving. Note that
\begin{equation}\label{e.Mn}
	M(n)
	\sim \lvert\log\delta(n)+\psi(n)\rvert.
\end{equation}
Now consider the map $g\eqdef f_0^{M(n)}\circ f_\xi^n\circ f_0^{N(n)}$ and note that it reverses orientation. Hence, there exists a point $y$ in the fundamental domain $[1/2,f_0(1/2))$ such that $g(y)=y$.
Note that by  the estimates of $N(n)$ and $M(n)$ in \eqref{e.Nn} and \eqref{e.Mn} and our choice of $\delta(n)$ in~\eqref{eq:defdeltan} and by \eqref{eq:limpsin} we have
\begin{equation}\label{eq:asyNnMn}
	\lim_{n\to\infty}\frac{N(n)+M(n)}{n}
	=0.
\end{equation}

We now consider the (invariant) measure $\mu_{n,\delta}$ supported on the periodic orbit of the point $Y=(\eta,y)$, where $\eta=(0^{N(n)}\xi_0\ldots\xi_{n-1}0^{M(n)})^\bZ$. It remains to show that this measure is close to $\mu$ in the weak$\ast$ topology provided that  $n$ was big. Note that we can write $\mu_{n,\delta}$ as
\[\begin{split}
	\mu_{n,\delta}=&\frac{N(n)}{N(n)+n+M(n)}\mu_1 +\\
	& 
		+\frac{n}{N(n)+n+M(n)}\frac1n\sum_{k=0}^{n-1}\delta_{F^{N(n)+k}(Y)} 
		+ \frac{M(n)}{N(n)+n+M(n)}\mu_2,
\end{split}\]
where $\mu_1$ and $\mu_2$ are some probability measures.
Note that  by~\eqref{eq:asyNnMn} the first and the last term converges to $0$ as $n$ tends to $\infty$. The second term is close to $\mu$ because $X$ was a $\mu$-generic point and the orbit piece $\{F^{N(n)}(Y),F^{N(n)+1}(Y),\ldots,F^{N(n)+n}(Y)\}$ is $\delta(n)e^{\psi(n)}$-close to the orbit piece $\{X,F(X),\ldots,F^{n}(X)\}$. Recalling~\eqref{eq:deltanasym}, this completes the proof in the affine case.

\medskip\noindent\textbf{General case:} In the nonaffine case the proof  goes similarly. Applying Lemma~\ref{lem:affineapprox}, we choose the number $\delta(n)$ in an appropriate way.
First note that instead of~\eqref{eq:estimprelim} by this lemma we have
\[
	 \frac{\dist(x_{i+1},B)}{\dist(x_i,B)}
	\le e^{\phi(i+1)-\phi(i)} e^{\Delta(\delta)}.
\] 
Arguing as above, let now
\[
	\delta(n)
	\eqdef \delta e^{-2\max\{\psi(n),\sqrt n\}} e^{-n\Delta(\delta e^{-\sqrt n})}.
\]
Observe that with this choice, for every $x_0\in[1-\delta(n),1)$ for every $i\in\{1,\ldots, n\}$ we have
\[
	\dist(x_i,B)
	\le \delta(n)e^{\phi(i)}e^{i\Delta(\delta e^{-\sqrt n})}
	\le \delta e^{-\sqrt n}
\]
provided that $\dist(x_j,B)\le\delta e^{-\sqrt n}$ for all $j\in\{0,\ldots,i-1\}$. By induction, we will get that for all $i\in\{0,\ldots, n\}$ we have
\[
	\dist(x_i,B)
	\le \delta e^{-\sqrt n}.
\]
Note that with the above definition of $\delta(n)$   the estimates of $N(n)$ and $M(n)$ in \eqref{e.Nn} and \eqref{e.Mn} remain without changes. And the rest of the proof is analogous to the partially affine case.
\end{proof}

We now prove the converse to Proposition~\ref{pro:measzeroisapprox}.

\begin{proposition}\label{pro:positivedist}
	For every $\mu\in\cM(\Lambda_{\rm ex})$	for which there exists a sequence $(\nu_k)_k\subset\cM(\Lambda_{\rm core})$ of measures which converge to $\mu$ in the weak$\ast$ topology we have $\chi(\mu)=0$.
\end{proposition}

The proof of the above proposition will be an immediate consequence of the following lemma. Recall the definition of $\Delta(\cdot)$ in~\eqref{eq:defDelta}.

\begin{lemma}
	There exist constants $K_1,K_2>0$ such that for every $\delta\in(0,1/2)$ and every measure $\nu\in\cM(\Lambda_{\rm core})$ we have
\[
	\lvert\chi(\nu)\rvert
	\le K_1\nu(\Sigma_2\times[\delta,1-\delta])+K_2\Delta(\delta).
\]	
\end{lemma}

\begin{proof}
Note that it is enough to prove the claim for $\nu\in\cM(\Lambda_{\rm core})$ being ergodic. Indeed, for  a general invariant measure $\nu\in\cM(\Lambda_{\rm core})$ with ergodic decomposition $\nu=\int\nu_\theta\,d\lambda(\nu_\theta)$, applying the above claim  to any (ergodic) $\nu_\theta$ in this decomposition we have
\[
	\chi(\nu)
	= \int \chi(\nu)\,d\lambda(\nu_\theta)
	\le  K_1\nu(\Sigma_2\times[\delta,1-\delta]) + K_2\Delta(\delta)
\]
with the analogous lower bound. 

Let us hence assume that $\nu\in\cM(\Lambda_{\rm core})$ is ergodic. 
Since $\nu$ is not supported on $\Lambda_{\rm ex}$, there exists $\delta'\in(0,\delta)$ such that $\nu(\Sigma_2\times [\delta',1-\delta'])>0$. 
Let $X=(\xi,x)$ be a generic point for $\nu$ satisfying $x\in[\delta',1-\delta']$ and consider the sequence of points $x_i\eqdef f_\xi^i(x)$ for $i\ge0$.
Since $\nu$ is ergodic, there are infinitely many $n\ge0$ such that  $x_n\in[\delta',1-\delta']$ and hence
\begin{equation}\label{eqdeltaprime}
	2\delta'
	\le \frac{\dist(x_n,B)}{\dist(x_0,B)}
	\le \frac{1}{2\delta'},
\end{equation}
where  $B=\{0,1\}$.
Because $X$ is a generic point, given any $\varepsilon$, for $n$ large enough we have 
\[
	\Big\lvert \frac1n\log\,\lvert(f_\xi^n)'(x)\rvert - \chi(\nu)\Big\rvert
	\le \varepsilon
\]	 
and also
\begin{equation}\label{eq:cardI1}
	\Big\lvert\frac1n\card\{i\in\{0,\ldots,n-1\}\colon x_i\in[\delta,1-\delta]\}
	-\nu(\Sigma_2\times[\delta,1-\delta])\Big\rvert
	\le \varepsilon.
\end{equation}

Applying Lemma~\ref{lem:affineapprox}, we have
\begin{gather*}
	e^{\omega(i+1)}e^{-\Delta(\delta)}
	\le \frac{\dist(x_{i+1},B)}{\dist(x_i,B)}
	\le e^{\omega(i+1)}e^{\Delta(\delta)}
	\quad\text{ if }x_i\in (0,\delta]\cup[1-\delta,1),\\
	e^{\omega(i+1)}K^{-1}
	\le \frac{\dist(x_{i+1},B)}{\dist(x_i,B)}
	\le e^{\omega(i+1)}K
		\quad\text{ if }x_i\not\in (0,\delta]\cup[1-\delta,1),
\end{gather*}
where $K>1$ is some universal constant  and
\[
	\omega(i)
	\eqdef \begin{cases}
	\log\,\lvert f'_{\xi_i}(0)\rvert  &\text{ if }x_{i-1}\in(0,\delta],\\
	\log\,\lvert f'_{\xi_i}(1)\rvert  &\text{ if }x_{i-1}\in[1-\delta,1),\\
	0 &\text{ otherwise}.
	\end{cases}
\]	
By a telescoping sum, we have
\[\begin{split}
	\frac{\dist(x_{n},B)}{\dist(x_0,B)}
	= \prod_{i=0}^{n-1}\frac{\dist(x_{i+1},B)}{\dist(x_i,B)}.
\end{split}\]
We split the index set $\{0,\ldots,n-1\}=I_1\cup I_2$ according to the rule that $x_i\in [\delta,1-\delta]$ for all $i\in I_1$ and $x_i\not\in [\delta,1-\delta]$ for all $i\in I_2$. 
Let
\[
	\phi(n)
	\eqdef\sum_{i=1}^n\omega(i)
\]	
and note that this function was also used in the previous proof, see~\eqref{eq:previous}.
By the above estimates, we hence have
\[
	e^{\phi(n)}K^{-\card I_1}	\cdot e^{-\Delta(\delta)\card I_2}
	\le \frac{\dist(x_{n},B)}{\dist(x_0,B)}
	\le e^{\phi(n)}K^{\card I_1}	\cdot e^{\Delta(\delta)\card I_2}.
\]
This implies
\begin{equation}\label{eq:estimephin}
	 e^{\phi(n)}
	\le \frac{\dist(x_{n},B)}{\dist(x_0,B)}
		K^{\card I_1}	\cdot e^{\Delta(\delta)\card I_2}.	
\end{equation}
Again applying Lemma~\ref{lem:affineapprox}, if $\dist(x_i,B)<\delta$ then we have
\[
	\lvert (f_{\xi_{i+1}})'(x_i)\rvert
	\le e^{\omega(i+1)} e^{\Delta(\delta)}
\]
 and   if $\dist(x_i,B)\ge\delta$ then we have
\[
	\lvert (f_{\xi_{i+1}})'(x_i)\rvert
	\le e^{\omega(i+1)} L
\]
for some universal $L>1$. Hence, decomposing the orbit piece $(x_i)_{i\ge0}^{n-1}$ as above into index sets $I_1$ and $I_2$, we obtain
\[
	\lvert (f_\xi^n)'(x)\rvert
	\le e^{\phi(n)}L^{\card I_1}\cdot e^{\Delta(\delta)\card I_2}
\]
with the analogous lower bound. 

By~\eqref{eq:cardI1}, we have $\card I_1\le n(\nu(\Sigma_2\times[\delta,1-\delta])+\varepsilon)$. 

Substituting the estimate for $e^{\phi(n)}$ in~\eqref{eq:estimephin} we obtain
\[
	\lvert (f_\xi^n)'(x)\rvert
	\le \frac{\dist(x_{n},B)}{\dist(x_0,B)}
		(KL)^{\card I_1}	\cdot e^{2\Delta(\delta)\card I_2}
	\le \frac{1}{2\delta'}	(KL)^{\card I_1}	\cdot e^{2\Delta(\delta)\card I_2},
\]
where we also used~\eqref{eqdeltaprime}. Hence
\[
	\frac1n\log\,\lvert (f_\xi^n)'(x)\rvert
	\le \frac1n\lvert\log(2\delta')\rvert
			+\log(KL)\big(\nu(\Sigma_2\times[\delta,1-\delta])+\varepsilon\big)
		+2\Delta(\delta),
\]
with the analogous lower bound. Since $\lvert\log\,\lvert (f_\xi^n)'(x)\rvert/n-\chi(\nu)\rvert\le\varepsilon$, passing $n\to\infty$ and then $\varepsilon\to0$ this ends the proof of the lemma.
\end{proof}
	
\begin{proof}[Proof of Theorem~\ref{thepro:measurespace}]
Item 1 is a well-known fact, see for example~\cite[Proposition 2 item (a)]{Sig:74}. 
This  fact implies that $\cM_{\rm erg}(\Lambda_{\rm ex})$ is a Poulsen simplex (see \cite{LinOlsSte:78} or in the particular case of the shift space \cite{Sig:77}).
Item 4 is then an immediate consequence from the facts that $\mu\mapsto\chi(\mu)$ is continuous and that the Dirac measure on $(0^\bZ,0)$ has Lyapunov exponent $\log\,f_0'(0)>0$ and the Dirac measure on $(0^\bZ,1)$ has Lyapunov exponent $\log\,f_0'(1)<0$ together with the fact that $\cM_{\rm erg}(\Lambda_{\rm ex})$ is path-connected. 

Item 2 follows from Proposition~\ref{pro:positivedist}.

Item 3 follows from Proposition~\ref{pro:measzeroisapprox}.

Let us now assume (H1), (H2'), (H3), and (H4). By Lemmas~\ref{l.c.expandingperiodic} and~\ref{l.c.contractingperiodic} there exist hyperbolic periodic points in $\Lambda_{\rm core}$ with positive and negative exponent, respectively. This proves item 5. By Proposition~\ref{pro:22} we can apply~\ref{pro:Katok}. This implies item 6. 
\end{proof}

%------------------------------------------------------------------------------------------------------
\section{The core measures}\label{s.core}	
%------------------------------------------------------------------------------------------------------

In this section we will investigate a bit further the topological structure of $\cM_{\rm erg}(\Lambda_{\rm core})$. The overall hypotheses are again (H1) and (H2), and we will discuss further additional conditions under which we are able to say more than in the previous sections.

\begin{proposition}\label{pro:Katok}
Assume that every pair of fiber expanding hyperbolic periodic orbits in $\Lambda_{\rm core}$ are homoclinically related.  Then the set $\cM_{\rm erg,>0}(\Lambda_{\rm core})$ is arcwise-connected. 
The analogous result holds true for fiber contracting hyperbolic periodic orbits in $\Lambda_{\rm core}$ and the set $\cM_{\rm erg,<0}(\Lambda_{\rm core})$.
\end{proposition}

A map $F$ whose fiber maps $f_0,f_1$ satisfy the hypotheses (H1), (H2'), (H3), and (H4) will satisfy the hypotheses of the above proposition. 

We will several times refer to a slightly strengthened version of~\cite[Proposition 1.4]{Cro:11} which, in fact, is contained in its proof in~\cite{Cro:11} and which can be seen as an ersatz of Katok's horseshoe construction (see~\cite[Supplement S.5]{KatHas:95}) in the $C^1$ dominated setting. We formulate it in our setting. Note that to guarantee that the approximating periodic orbits are indeed contained in $\Sigma_2\times(0,1)$ it suffices to observe that in the approximation arguments one can consider any sufficiently large (in measure $\mu$) set and hence restrict to points which are uniformly away from the ``boundary'' $\Sigma_2\times\{0,1\}$. Indeed, the projection to $[0,1]$ of the support of $\mu$ can be the whole interval $[0,1]$ but it does not ``concentrate'' in $\{0,1\}$.

\begin{lemma}\label{lem:sylv}
Let $\star\in\{<0,>0\}$ and $\mu\in\cM_{\rm erg,\star}(\Lambda_{\rm core})$. 

Then for every $\rho\in(0,1)$ there exist $\alpha>0$ and a set $\Gamma_\rho\subset\Sigma_2\times(2\alpha,1-2\alpha)$ and a number  $\delta=\delta(\rho,\mu)>0$ such that $\mu(\Gamma_\rho)>1-\rho$ and for every point $X\in\Gamma_\rho$ there is a sequence $(p_n)_n\subset\Sigma_2\times(\alpha,1-\alpha)$ of hyperbolic periodic points such that:
\begin{itemize}
\item $p_n$ converges to $X$ as $n\to\infty$;
\item the invariant measures $\mu_n$ supported on the orbit of $p_n$ are contained in $\cM_{\rm erg,\star}(\Lambda_{\rm core})$ and converge to $\mu$ in the weak$\ast$ topology;
\end{itemize}
\end{lemma}

\begin{proof}[Proof of Proposition~\ref{pro:Katok}]
Similar results were shown before, though in slightly different contexts (see~\cite{GorPes:17} and \cite[Theorem 3.2]{DiaGelRam:17}). For completeness, we sketch the proof. 

Assume that $\mu^0,\mu^1\in\cM_{\rm erg,>0}(\Lambda_{\rm core})$. By Lemma~\ref{lem:sylv}, $\mu^i$ is accumulated by a sequence of hyperbolic periodic measures $\nu_n^i\in\cM_{\rm erg,>0}(\Lambda_{\rm core})$ supported on the orbits of fiber expanding hyperbolic periodic points $P_n^i\in\Lambda_{\rm core}$, $i=0,1$. Since, by hypothesis, $P_1^0$ and $P_1^1$ are homoclinically related, there exists a horseshoe $\Gamma_1^{0,1}\subset\Lambda_{\rm core}$ containing these two points. Hence, since $\cM(\Gamma_1^{0,1})$ is a Poulsen simplex \cite{LinOlsSte:78,Sig:77}, there is a continuous arc $\mu_0\colon[1/3,2/3]\to\cM_{\rm erg,>0}(\Gamma_1^{0,1})\subset \cM_{\rm erg,>0}(\Lambda_{\rm core})$ joining the measures $\nu_1^0$ and $\nu_1^1$. For any pair of measures $\nu_n^0,\nu_{n+1}^0$, the same arguments apply and, in particular, there exists a continuous arc $\mu_n^0\colon[1/3^{n+1},1/3^n]\to\cM_{\rm erg,>0}(\Lambda_{\rm core})$ joining the measure $\nu_n^0$ with $\nu_{n+1}^0$. 
Using those arcs and concatenating  their domains (or appropriate parts of), we can construct an arc $\bar \mu_n^0\colon[1/3^{n+1},1/3]$ joining $\nu_{n+1}^0$ and $\nu_{1}^0$.
The same applies to the measures $\nu_n^1$, defining arcs $\bar \mu_n^1 \colon[1-1/3^n,2/3]\to\cM_{\rm erg,>0}(\Lambda_{\rm core})$ 
joining $\nu_{n+1}^1$ and $\nu_{1}^1$. 
 Defining $\mu_\infty|_{(0,1)}\colon(0,1)\to\cM_{\rm erg,>0}(\Lambda_{\rm core})$ by concatenating (appropriate parts of) the domains of those arcs,  we  complete the definition of the arc 
$\mu_\infty$ by letting $\mu_\infty(0)=\lim_{n\to\infty}\bar \mu_n^0(1/3^n)$ 
and $\mu_\infty(1)=\lim_{n\to\infty}\bar \mu_n^1(1-1/3^n)$,
joining $\mu^0$ and $\mu^1$. Note that in the last step we assume that $\mu^1,\mu^2$ do not belong to the image of $\mu_\infty$, if one of these measures belongs it is enough to cut the domain of definition of $\mu_\infty$ appropriately.
\end{proof}

%------------------------------------------------------------------------------------------------------
\appendix
\section{Transitivity and homoclinic relations. Proof of Proposition~\ref{pro:22}}
 \label{a.apendix}
%------------------------------------------------------------------------------------------------------

In this appendix we prove Proposition~\ref{pro:22}. Hence, we will always assume that hypotheses (H1), (H2'), (H3), (H4) are satisfied.

%------------------------------------------------------------------------------------------------------
\subsection{The underlying IFS}
\label{a.underlying}
%------------------------------------------------------------------------------------------------------
Studying the iterated function system (IFS) associated to the maps $\{f_0,f_1\}$, we use the following notations.
Every sequence $\xi=(\ldots\xi_{-1}.\xi_0\xi_1\ldots)\in\Sigma_2$ is given by $\xi=\xi^-.\xi^+$, where $\xi^+\in\Sigma_2^+\eqdef\{0,1\}^{\bN_0}$ and $\xi^-\in\Sigma_2^-\eqdef\{0,1\}^{-\bN}$.
Given  \emph{finite} sequences $(\xi_0\ldots \xi_n)$ and  $(\xi_{-m}\ldots\xi_{-1})$, we let
\[\begin{split}
    f_{[\xi_0\ldots\,\xi_n]}
    &\eqdef f_{\xi_n} \circ \cdots \circ f_{\xi_1}\circ f_{\xi_0} 
    \quad\text{ and }\quad\\
    f_{[\xi_{-m}\ldots\,\xi_{-1}.]}
    &\eqdef  (f_{\xi_{-1}}\circ\ldots\circ f_{\xi_{-m}})^{-1}
    =(f_{[\xi_{-m}\ldots\,\xi_{-1}]})^{-1}.
\end{split}\]

%------------------------------------------------------------------------------------------------------
\subsubsection{Expanding itineraries}
%------------------------------------------------------------------------------------------------------

Under hypotheses (H1) and (H4) there are a positive number $\varepsilon$ arbitrarily close to $0$, a positive integer $N(\varepsilon)$, and fundamental domains
$I_0(\varepsilon)= [\varepsilon, f_0(\varepsilon)]$ and  $I_1(\varepsilon)= [1-\varepsilon, f_0(1-\varepsilon)]$ of $f_0$ having the following properties%
\footnote{Just note that, by the mean value theorem, there is $z\in I_0(\varepsilon)$ with $(f_0^{N(\varepsilon)})'(z) =|I_1(\varepsilon)|/|I_0(\varepsilon)|$,
that by monotonicity of the derivative of $f_0'$ we have $(f_0^{N(\varepsilon)})'(f_0(z)) \ge \lambda \, |I_1(\varepsilon)|/(\beta\, |I_0(\varepsilon)|)$ and that $(f_0^{N(\varepsilon)})'(x)\ge (f_0^{N(\varepsilon)})'(f_0(z))$ for all $x\in I_0(\varepsilon)$, and that for small $\varepsilon$ we have $|I_0(\varepsilon)|\simeq (\beta -1)\, \varepsilon$ and $|I_1(\varepsilon)|\simeq (1-\lambda)\, \varepsilon$.}
\begin{equation}\label{e.expansion}
	f_0^{N(\varepsilon)} (I_0(\varepsilon))=I_1(\varepsilon) 
	\quad \mbox{and}\quad
	(f_0^{N(\varepsilon)})^\prime (x)
	\ge \lambda^{-1} \varkappa >1
	\quad\text{for all }x\in I_0 (\varepsilon). 
\end{equation}
In what follows we fix small $\varepsilon>0$ satisfying the above conditions and for simplicity we write $I_0$, $I_1$, and $N$ instead of $I_0(\varepsilon)$, $I_1(\varepsilon)$, and $N(\varepsilon)$.

Our construction now is analogous to the one in \cite{DiaGelRam:13}. We sketch the main steps for completeness. Assuming additionally (H2'), given an interval $H\subset f_0^{-1}(I_0)\cup I_0$ we let $N(H)=N$ if $H\subset I_0$ and
$N(H)=N+1$ otherwise and consider the interval $f_{[0^{N(H)}1]}(H)$. By construction, this interval is
contained in $[\delta (\varepsilon),\varepsilon]$, where $\delta(\varepsilon)=1-f_0^2 (1-\varepsilon)$. Note that, by construction,
$\delta (\varepsilon)<\varepsilon$.
Therefore there is a first $M(H)$ such that
$$
	f_{ [0^{N(H)}10^{M(H)} ]}(H)\cap (\varepsilon,f_0(\varepsilon)]\ne\emptyset.
$$
 The \emph{expanded successor} of $H$ is the interval $H'\eqdef f_{ [0^{N(H)}10^{M(H)} ]}(H)$. The \emph{expanding return sequence} of $H$ is the finite sequence $0^{N(H)}10^{M(H)}$ By construction  the interval $H'$  intersects the interior of $I_0$ and is contained in $[\delta(\varepsilon), f_0(\varepsilon)]$. Also observe that there is $M$ such that $M(H)\in \{1,\ldots, M\}$ for every subinterval $H$ in $ f_0^{-1}(I_0)\cup I_0$. The following lemma justifies our terminology expanded successor.

\begin{lemma}[Expanding itineraries {\cite[Lemma~2.3]{DiaGelRam:13}}]\label{l.expandedsucessor}
For every closed subinterval $H$ of $f_0^{-1}(I_0)\cup I_0$ and every $x\in H$ it holds
$$
\big| \big(f_{[0^{N(H)}10^{M(H)}]}\big)'(x) \big| \ge \varkappa>1.
$$
\end{lemma}

\begin{proof}
By \eqref{e.expansion} and the choice of $H$ we have
$\big|(f_{[0^{N(H)}]})^\prime (x) \big|\ge \varkappa$ for all $x\in H$. The assertion follows noting that $f_1(x)=1-x$,
$(f_{[0^{N(H)}1]})(x)\in [0,\varepsilon]$ if $x\in H $,
 and 
$f_0^\prime(y)>1$ if $y\in [0,\varepsilon]$. 
\end{proof}

Lemma~\ref{l.expandedsucessor} and an inductive argument immediately implies the following:

\begin{lemma}[{\cite[Lemma~2.3]{DiaGelRam:13}}]\label{l.expandedcovering}
For every closed subinterval $H$ of $f_0^{-1}(I_0)\cup I_0$ there is a finite sequence $(\xi_0\ldots \xi_{\ell(H)})$ such that
\begin{enumerate}
\item
$\big| \big( f_{ [\xi_0\ldots \,\xi_{\ell(H)}]}\big)'(x) \big|\ge \varkappa$ for every $x\in H$ and
\item
$f_{[\xi_0\ldots\, \xi_{\ell(H)]}} \supset f_0^{-1}(I_0)$.
\end{enumerate}
\end{lemma}

\begin{proof}
Write $H_0$ and let $H_1=H_0'$ be its expanding successor. We argue recursively, if $H_1$ contains $f_0^{-1}(I_0)$ we stop the recursion, otherwise we observe that $|H_1|\ge \varkappa |H_0|$ and consider the expanding successor $H_2=H'_1$.
Since $H_i\ge \varkappa^i |H_0|$ there is a first $i$ such that $H_i$ contains $f_0^{-1}(I_0)$. We let $(\xi_0\ldots \xi_{\ell(H)})$ be the concatenation of the successive expanding returns. 
\end{proof}

Given a set $H\subset [0,1]$ denote its {\emph{forward orbit}} by the IFS by
$$
\cO^+(H)\eqdef \bigcup_{k\ge 0}
 \, 
 \bigcup_{(\xi_0\ldots\, \xi_k) \in \{0,1\}^{k+1}   }  f_{[\xi_0\ldots\, \xi_k]} (H).
$$
 A special case occurs when the set $H$ is a point.

\begin{lemma}\label{l.c.expandingperiodic}
For every point $p\in (0,1)$ there are a small neighborhood $I(p)$ of $p$ and a finite sequence $\eta_{0}\ldots \eta_{r}$, $r=r(I(p))$, such that
\begin{enumerate}
\item[(1)]
$f_{[\eta_{0}\ldots \eta_{r}]}(I(p))\supset I(p)$,
\item[(2)]
$\big| \big(f_{[\eta_{0}\ldots \eta_{r}]})^\prime (x)\big| >1$ for all $x\in I(p)$, and
\item[(3)]
$\cO^+(I(p))=(0,1)$.
\end{enumerate}
\end{lemma}

\begin{proof}
Without loss of generality (considering some backward iterate of $p$ and possibly shrinking $\varepsilon$) we can assume that $p\in (f^{-1}_0(\varepsilon), \varepsilon]$. Let us suppose, for simplicity that $p\ne \varepsilon$ (the case $p=\varepsilon$ would require a small additional step). In such a case we can take $I(p)\subset (f^{-1}_0(\varepsilon), \varepsilon)$ and apply Lemma \ref{l.expandedcovering} to $H=I(p)$. This gives conditions (1) and (2) in the lemma. To get (3) note that we can assume that $f_{[\eta_{0}\ldots \eta_{r}]}(I(p))$ covers the
fundamental domain
$f_0^{-1}(I_0)$. Therefore 
$$
\bigcup_{j\ge 0}  f_0^j (f_{[\eta_{0}\ldots \eta_{r}]}(I(p))) \supset (\varepsilon, 1)
$$ 
and 
$$
\bigcup_{j\ge 0}  f_1\circ f_0^j (f_{[\eta_{0}\ldots \eta_{r}]}(I(p))) \supset (f_1(1), f_1(\varepsilon))=
(0,f_1(\varepsilon)).
$$ 
Since $f_1(\varepsilon)>\varepsilon$ the claim follows.
\end{proof}

%------------------------------------------------------------------------------------------------------
\subsubsection{Contracting itineraries}\label{a.contracting}
%------------------------------------------------------------------------------------------------------

For the contracting itineraries we will now in particular focus on (H3), which plays the role of (H4) in the previous subsection.
Recall that $c\in (0,1)$ is given by the condition $f_0^\prime (c)=1$. Note that, since $f_0'$ is decreasing, we have $f_0^\prime(f_0(c))<1$ and hence
\begin{equation}\label{e.defc}
\upsilon \eqdef \frac{1}{f_0^\prime(f_0(c))}>1.
\end{equation}
In what follows, for notational simplicity let $g_0\eqdef f_0^{-1}$ and $g_1\eqdef
f_1^{-1}$($=f_1$) and below consider the IFS generated by $\{g_0,g_1\}$.

Next lemma is a variation of \cite[Lemma 2.6]{DiaGelRam:13}, where an important difference is that
in our case $g_1$ is not expanding.

\begin{lemma}[Contracting itineraries]\label{l.contracting}
Let $H$ be a closed subinterval of $[c,f_0^2(c)]$. Then there are a subinterval
$H_0$ of $H$ and a sequence $\xi_0\dots \xi_k$ such that
$$
	g_{[\xi_0\dots \xi_k]}(H_0)\supset [f_0(c), f_0^2(c)]
	\quad\mbox{and} \quad
	|g^\prime_{[\xi_0\dots \xi_k]}(x)|\ge \upsilon
%{\varkappa}^{(k+1)/3} 
	\quad\mbox{ for every $x\in H_0$}.
$$
\end{lemma}

\begin{proof}
Note that for every $x\in  [f_0(c), f_0^2(c)]$ it holds $|g_{[01]}^\prime (x)| \ge \upsilon$. Condition $f_1 \circ f_0^2(c)> f_0^2(c)$ implies that $g_{[01]}(x)> f_0^2(c)$. Thus,  there is a first $i\ge 0$ such that $g_{[010^i]}(x)\in   [f_0(c), f_0^2(c)]$.
Note that $|g_{[010^i]}'(x)|\ge \upsilon$. Now the result follows arguing as in  Lemma~\ref{l.expandedcovering}.
\end{proof}

Define the {\emph{backward orbit}} $\cO^-(\cdot)$ by the IFS of a set in the natural way. Arguing as in the expanding case, we have the following version of Lemma~\ref{l.c.expandingperiodic}.

\begin{lemma}
\label{l.c.contractingperiodic}
For every point $p\in (0,1)$ there are a small neighborhood $J(p)$ of $p$ and a finite sequence $(\nu_{0}\ldots \nu_{r})$, $r=r(J(p))$, such that
\begin{enumerate}
\item[(1)] $f_{[.\nu_{0}\ldots \nu_{r}]}(J(p))\supset J(p)$,
\item[(2)] $\big| \big(f_{[.\nu_{0}\ldots \nu_{r}]})^\prime (x)\big| >1$ for all $x\in J(p)$, and
\item[(3)] $\cO^-(J(p))=(0,1)$.
\end{enumerate}
\end{lemma}

%------------------------------------------------------------------------------------------------------
\subsubsection{Almost forward and backward minimality}%\label{minimality}
%------------------------------------------------------------------------------------------------------

\begin{corollary}[Almost minimality]\label{c.minimality}
For every $x\in (0,1)$ the sets $\mathcal{O}^+(x)$ and  $\mathcal{O}^-(x)$ are both dense in $[0,1]$.
\end{corollary}

\begin{proof}
Fix  any $x\in (0,1)$. To prove the backward minimality fix $p\in (0,1)$ and an arbitrarily small neighborhood $J(p)$ of it. By Lemma~\ref{l.c.contractingperiodic} item (3) we have that $x\in \mathcal{O}^-(J(p))$ and hence $J(p) \cap \mathcal{O}^+(x)\ne\emptyset$. 
The proof of the forward minimality is analogous using Lemma~\ref{l.c.expandingperiodic} item (3).
\end{proof}

\subsection{Transitive dynamics. Homoclinic relations}

To prove that $F$ is topologically transitive, we use the notion of a 
{\emph{homoclinic class}} adapted to the skew product setting. For that we need some definitions. Observe that if  $P=(\xi,p)$ is a periodic point of
$F$ of period $k+1$, then $\xi =(\xi_0\ldots \xi_{k})^{\bZ}$ and $f_{[\xi_0\ldots\,\xi_k]}(p)=p$. Note that
\[
	\frac{1}{k+1}\log\,\lvert f_{[\xi_0\ldots\,\xi_k]}^\prime (p)\rvert 
	= \chi(P).	
\]
If $\chi (P)\ne 0$ then we call $P$ {\emph{(fiber) hyperbolic}}. There are two types of such points: if $\chi (P)>0$ then we call $P$ \emph{(fiber) expanding}, otherwise  $\chi (P)<0$  and we call $P$ \emph{(fiber) contracting}.
We denote by $\Per_{\mathrm{hyp}}(F)$ the set of all fiber hyperbolic periodic points of $F$ and by $\Per_{>0}(F)$ and $\Per_{<0}(F)$ the (fiber) expanding and (fiber) contracting periodic points, respectively. Clearly, 
 $\mathrm{Per}_{\mathrm{hyp}}(F)= \mathrm{Per}_{>0}(F)\cup
 \mathrm{Per}_{<0}(F)$.
 Given a fiber hyperbolic periodic point $P$ we consider the {\emph{stable and unstable sets}}
 of its orbit $\cO(P)$ denoted by $W^s\big( \cO (P_1), F\big)$ and $W^u\big( \cO (P_1), F\big)$. 
 
 Two periodic points $P_1,P_2 \in \mathrm{Per}_{\mathrm{hyp}}(F)$ of the same type of hyperbolicity 
 (that is, either both points are fiber expanding or both are fiber contracting)
with different orbits 
$\cO(P_1)$ and $\cO(P_2)$ are \emph{homoclinically related}
if the stable and unstable sets of their orbits intersect cyclically:
$$
	W^s\big( \cO (P_1), F\big) \cap  W^u\big( \cO (P_2), F\big) \ne\emptyset
	\,\,\mbox{and}\,\,
	W^u\big( \cO (P_1), F\big) \cap  W^s\big( \cO (P_2), F\big) \ne\emptyset.
$$
A point $X\not\in \cO(P)$ is a \emph{homoclinic point} of $P$ if 
$$
	X\in W^s\big( \cO (P), F\big) \cap  W^u\big( \cO (P), F\big).
$$
Observe that our definitions  do not involved any transversality assumption (indeed in our context of a skew product such a transversality does not make sense, see also \cite[Section 3]{DiaEstRoc:16} for more details on homoclinic relations for skew products).
However, due to the fact that the maps $f_0$ and $f_1$ have no critical points, the homoclinic points behave as the transverse ones in the differentiable
setting.  

The \emph{homoclinic class} $H(P)$ of a fiber hyperbolic periodic point $P$  is the closure of the orbits of the periodic points  of the same type as $P$ which 
are homoclinically related to $P$. As in the differentiable setting, the
set $H(P)$  coincides with the closure of
the homoclinic points of $P$. This set is transitive. 

Let us introduce some notation. For $\star \in \{<0,>0\}$, define
$$
\mathrm{Per}_{\rm{core},\star}(F)\eqdef
\mathrm{Per}_\star(F)
 \cap \Lambda_{\rm{core}}
 \quad 
 \mbox{and}
 \quad
 \mathrm{Per}_{\rm{ex},\star}(F)\eqdef
\mathrm{Per}_\star(F)
 \cap \Lambda_{\rm{ex}}.
 $$

\begin{proposition}[Homoclinic relations]\label{pro:homrel}
Let $\star \in \{<0,>0\}$.
\begin{enumerate}
\item[(1)]
Every pair of points $R_1,R_2 \in \mathrm{Per}_{\rm{core},\star}(F)$
are homoclinically related. 
\item[(2)]
Every pair of points $R_1,R_2 \in \mathrm{Per}_{\rm{ex},\star}(F)$
are homoclinically related and their homoclinic classes coincide with $\Lambda_{\rm{ex}}$.
\item[(3)]
No point in $\mathrm{Per}_{\rm{core},\star}(F)$ is homoclinically related to any point of $\mathrm{Per}_{\rm{ex},\star}(F)$.
\item[(4)]
The set $\Sigma_2\times [0,1]$ is the homoclinic class of any $R\in \mathrm{Per}_{\rm{core},\star}(F)$. As a consequence, the set $\mathrm{Per}_{\rm{core},\star}(F)$ is dense in  $\Sigma_2\times [0,1]$. 
\end{enumerate}
\end{proposition}

\begin{proof}
As the arguments in this proof are similar to the ones in \cite[Section 2]{DiaGel:12} we willl just sketch them. We prove (1) for fiber contracting periodic points only. Fix $P= ((\xi_0\ldots \xi_k)^\mathbb{Z}, p)$ and $R= ((\eta_0\ldots \eta_\ell)^\mathbb{Z}, r)$, $r,p\in (0,1)$. Take an open interval $I(r)$ containing $r$ such that $I(r)\subset W^s_{\mathrm{loc}} (r, f_{[\eta_0\ldots \eta_\ell]})$. By Corollary~\ref{c.minimality}, there is $\rho_0\ldots \rho_m$ such that $f_{[\rho_0\ldots \rho_m]} (p)\in I(r)$.
Take $X= ( (\xi_0\ldots \xi_k)^{-\NN}. \rho_0\ldots \rho_m (\eta_0\ldots \eta_\ell)^\NN,p)$.
By construction, $X\in W^u(\mathcal{O}(P),F) \cap W^u(\mathcal{O}(R),F)$. Reversing the roles of $P$ and $R$ we obtain a point in 
$W^u(\mathcal{O}(R),F) \cap W^s(\mathcal{O}(P),F)$, proving that $P$ and $R$ are homoclinically related.

The proof of (2) is an immediate consequence of the fact that $F_{\Lambda_{\rm ex}}$ can be seen as an ``abstract horseshoe''. 

To prove item (3) note that $\cO^\pm(0),\cO^\pm(1)\subset\{0,1\}$. This prevents any  periodic point with fiber coordinate $0$ or $1$ to be homoclinically related to points in $\Lambda_{\mathrm{core}}$. 

We prove item (4) for expanding points only.
Fix an expanding periodic point $R\in \Lambda_{\mathrm{core}}$. 
Consider any point $X=(\xi,x)$, $x\in (0,1)$. Fix $m\ge$ and $\delta>0$ and consider $(\xi_{-m}\ldots \xi_m)$ and $I(\delta)=
(x'-\delta,x'+\delta)$, where $x'= f_{[\xi_{-m}\ldots \xi_{-1}.]}(x)$. Consider now $I'(\delta)=f_{[\xi_{-m}\ldots \xi_0\ldots \xi_m]} (I(\delta))$.
Applying Lemma~\ref{l.c.expandingperiodic}  to $I(\delta)'$, we get $\eta_0\ldots \eta_r$ such that
$f_{[\eta_0\ldots\eta_r]} (I(\delta)')$ covers $I(\delta)$ and 
$f_{[\xi_{-m}\ldots \xi_0\ldots \xi_m\eta_0\ldots \eta_r]}$ is expanding on $I(\delta)'$. This provides an expanding periodic point $P_{\delta,m}$ close to $X$. 
Note that $P_{\delta,m}\to X$ as $\delta\to 0$ and $m\to \infty$.
By item (1) this point is homoclinically related to $R$. As a consequence,  we have $X\in H(R,F)$.
\end{proof}

%------------------------------------------------------------------------------------------------------
\subsection{The parabolic case}
\label{a.parabolic}
%------------------------------------------------------------------------------------------------------

In this section we will prove Theorem~\ref{theoremA8a}. For that we see how the constructions above can be modified to construct examples where the set $\Lambda$ has an ergodic measure of maximal entropy which is nonhyperbolic. For this we modify the map $f_0$  satisfying conditions (H1), (H3), and (H4) to get a new map $\tilde f_0$ such that the points $0$ and $1$ are parabolic ($0$ is repelling and $1$ is attracting) and consider the skew product $\widetilde F$ associated to $\tilde f_0$ and $f_1(x)=x-1$.

\begin{proof}[Proof of Theorem~\ref{theoremA8a}]
We start with a map $f_0$ satisfying hypotheses (H1), (H3), and (H4) 
and consider exactly as in Appendix~\ref{a.underlying} the fundamental domains
$I_0(\epsilon)=[\varepsilon,f_0(\varepsilon)]$ and $I_1(\epsilon)=[1-\varepsilon, f_0(1-\varepsilon))$ (for small $\varepsilon>0$) 
and the natural number $N(\varepsilon)$ with $f_0^{N(\varepsilon)} (I_0(\varepsilon))=I_1(\varepsilon)$. Note that the estimate in \eqref{e.expansion}
holds. We define for a subinterval $H$ of $f_0^{-1}(I_0(\varepsilon)) \cup I_0(\varepsilon)$ the number $N(H)\in \{N(\varepsilon), N(\varepsilon)+1\}$. 
Similarly we define $M(H)\in \{0,\dots, M\}$ ($M$ independent of $H$). 

Assume not that $f_1$ satisfies (H2'). Let 
$$
 a_1 \eqdef\min\{f_0^{-1}(\varepsilon),1-f_0^{N(\varepsilon)+1} (\varepsilon)\},
 \quad
 b_1\eqdef f_0^{N(\varepsilon)+1} (\varepsilon).
 $$
Note also that the definition of the expanding successors only involves iterates in the set
$$
[f_0^{-1}(\varepsilon), f_0^{N(\varepsilon)+1} (\varepsilon)] \cup f_1 \big([f_0^{-1}(\varepsilon), f^{N(\varepsilon)+1} (\varepsilon)]\big)=
[a_1,b_1]\subset [\delta, 1-\delta],
 $$
 for some small $\delta>0$. We now fix very small $\tau\ll\delta$ and consider a new map $\tilde f_0$ such that
 \begin{itemize}
\item [(i)] $\tilde f_0=f_0$ in $[\delta,1-\delta]$, 
\item [(ii)] $(\tilde f_0)'(0)=1$ and $0$ is repelling, 
\item[(iii)] $(\tilde f_0)'(1)=1$ and $1$ is attracting, $\tilde f_0$ has no fixed points in $(0,1)$,
\end{itemize}
 see Figure~\ref{Fig:parab}.
Note that for this new map $\tilde f_0$ we can define expanding returns in $I_0(\varepsilon)$ as before. Note also that every point $x\in (0,1)$ has some forward and some backward iterate in $I_0(\varepsilon)$  by the IFS associated to $\{\tilde f_0, f_1\}$ (here we use that $0$ is repelling,  $1$ is attracting and $\tilde f_0$ has no fixed points in $(0,1)$. We now have versions of Lemmas~\ref{l.expandedsucessor}, \ref{l.expandedcovering}, and \ref{l.c.expandingperiodic} for the IFS associated to $\{\tilde f_0, f_1\}$. This concludes the part corresponding to the expanding itineraries.
 
\begin{figure}[h] \begin{overpic}[scale=.35,bb=0 0 330 330]{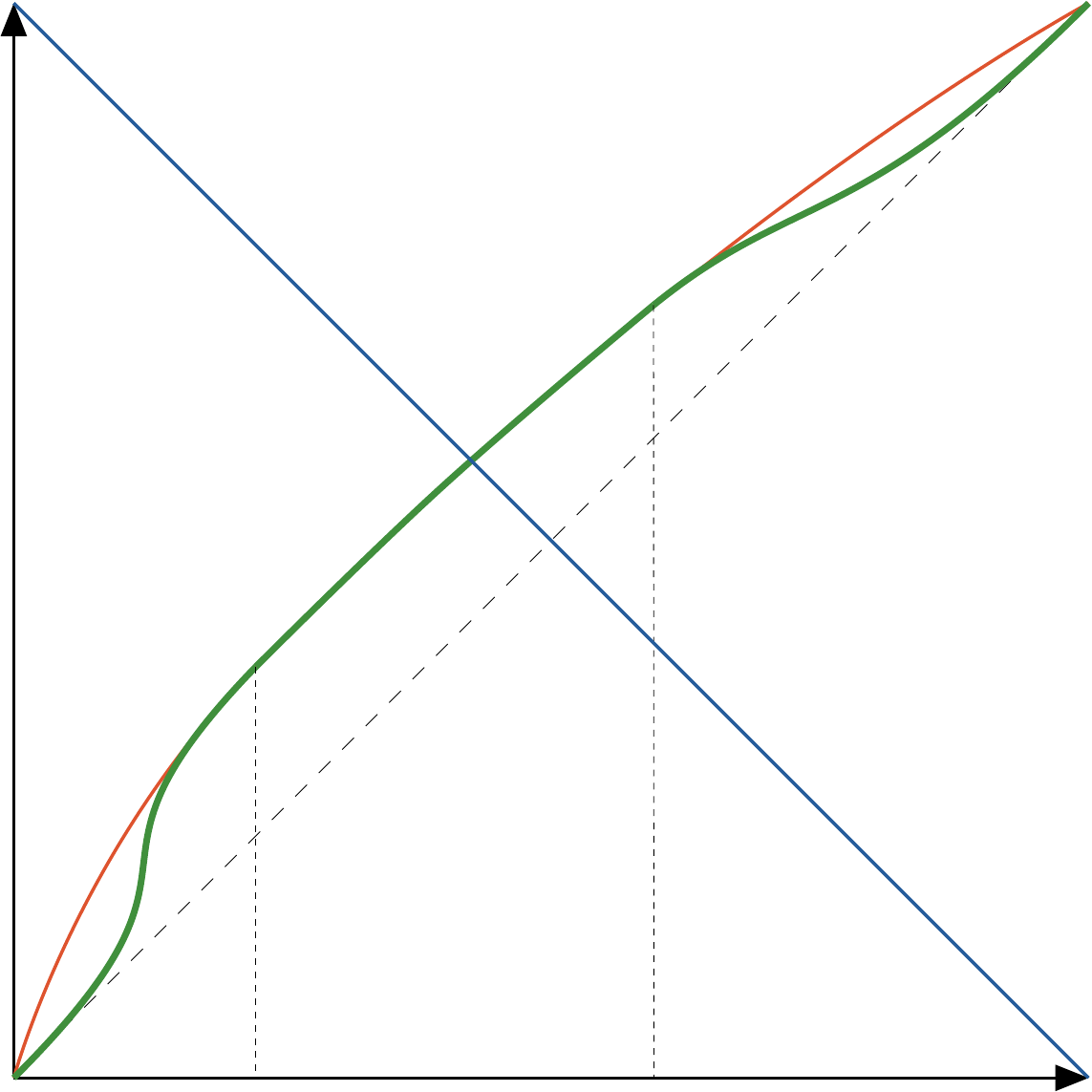}
 \put(73,92){$f_0$}
 \put(85,75){$\tilde f_0$}
 \put(80,28){$f_1$}
 \put(0,-10){$0$}
 \put(96,-10){$1$}
 \put(20,-10){$\delta$}
 \put(50,-10){$1-\delta$}
\end{overpic}
\caption{Fiber maps: The parabolic case}
\label{Fig:parab}
\end{figure}
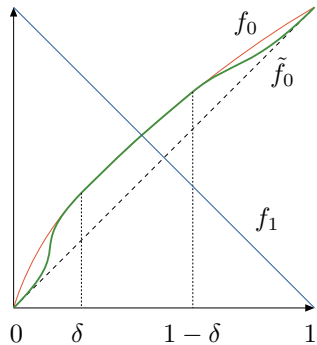

It remains to check that the arguments corresponding to the contracting itineraries in Appendix~\ref{a.contracting} also hold.
 Recall the definition of the point $c$ in hypothesis (H3), see also \eqref{e.defc}.  
 Let $g_0=f_0^{-1}$ and $g_1=f_1^{-1}$. 
 Note that if $\delta$ is small we can assume that 
 $\delta< g_0(c)< 1-g_0(c)]<1-\delta$.
Note that for closed intervals $H\subset [c,f_0^2(c)]$ the definition of their expanding successor only involves
 iterations in the set
 $[g_0(c), 1-g_0(c)]$. Since in this interval $\tilde g_0=g_0$ we obtain versions of Lemmas~\ref{l.contracting} and
\ref{l.c.contractingperiodic} for  the IFS associated to $\{\tilde g_0, g_1\}$. In the same way we recover Corollary~\ref{c.minimality}
for the  IFS associated to $\{\tilde f_0, f_1\}$.

We can now consider the skew product $\widetilde F$ associated to $\tilde f_0, f_1$ and prove Proposition~\ref{pro:homrel} for $\tilde F$,
obtaining, in particular, that the set $\Sigma_2\times [0,1]$ is a homoclinic class of $\tilde F$.
By Theorem~\ref{the:1} the unique measure $\mu_{\rm max}^{\rm ex}$ of maximal entropy $\log2$ in $\cM_{\rm erg}(\Lambda_{\rm ex})$ 
is nonhyperbolic. 
\end{proof}

\subsection{Nontransitive case with a unique measure of maximal entropy}\label{app:bad}

In this section we prove Theorem~\ref{theoremA8b} by presenting an example  which is not transitive and for which there exists just one measure of maximal entropy, which is nonhyperbolic. This measure is supported on $\Lambda_{\rm ex}$ and there is no measure of maximal entropy in $\Lambda_{\rm core}$.

\begin{proof}[Proof of Theorem~\ref{theoremA8b}]
Let us consider a $C^1$ orientation preserving homeomorphism $\phi \colon \bR \to (0,1)$ satisfying
\begin{equation}\label{eq:sym}
\phi(y) = 1-\phi(-y)
\end{equation}
and
\[
\lim_{y\to\infty} \frac {\phi'(y+1)} {\phi'(y)} \eqdef 1 \in (0,\infty).
\]
(For example, $\phi(y) = \frac 1 \pi \arctan y + \frac 12$ satisfies the conditions above.) Now define $f_0\colon[0,1]\to[0,1]$
\[
	f_0(x) \eqdef
				\begin{cases} 
				\phi(\phi^{-1}(x)+1)&\text{ if }x\in(0,1),\\
				0&\text{ if }x=0,\\
				1&\text{ if }x=1,
				\end{cases}
\]
and $f_1\colon[0,1]\to[0,1]$ by $f_1(x)=1-x$.
Note that $f_0$ is a $C^1$ map which satisfies 
\begin{equation}\label{eq:ufffa}
	f_0'(0)=1
	= f_0'(1).
\end{equation}
Moreover, note that $ f_0\circ\phi=\phi\circ\theta$, where $\theta\colon\bR\to\bR$ denotes the unit translation on the real line defined by $\theta(y)\eqdef y+1$. The symmetry assumption \eqref{eq:sym}  means that $ f_1\circ\phi=\phi\circ\gamma$ where $\gamma\colon\bR\to\bR$ is defined by $\gamma(y)\eqdef -y$. 
Observe that $\phi(-\phi^{-1}(x))=1-x$ implies 
\begin{equation} \label{eqn:comm}
f_0 f_1 = f_1 f_0^{-1},
\end{equation}
that is, $f_0$ is conjugate to its inverse by $f_1$. This provides us fiber maps $f_0,f_1$ satisfying item 1. and 2. in the theorem.

\begin{proposition}
 $F$   is not topologically transitive.
\end{proposition}

\begin{proof}
It suffices to prove that for any $x\in [0,1]$ the set $\cO^+(x)\eqdef\{f_{[\omega_1\ldots\,\omega_n]}(x)\colon \omega_i\in \{0,1\}, i\in\{1,\ldots,n\}\}$ is not dense in $[0,1]$. This is obvious if $x\in\{0,1\}$. 
Thus, in what follows, we let $x\in(0,1)$.

Given $n\in\bN$, consider some finite sequence $(\omega_1\ldots\omega_n)\in\{0,1\}^n$. First, recall that $f_1^2(x)=x$ we can replace this sequence by one in which we eliminated all blocks $11$. Hence, without loss of generality, we can assume that the sequence $(\omega_1\ldots\omega_n)$ does not contain two consecutive $1$s.
Assume first that this sequence contains an even number of symbols $1$, that is, we can divide it into a finite number of pieces of the form $0^k10^\ell 10^m$. By \eqref{eqn:comm}, $f_{[0^k10^\ell 10^m]}(x) = f_0^{k+m-\ell}(x)$. Hence, we have $f_{[\omega_1\ldots\omega_n]}(x)=f_0^j(x)$ for some integer $j$.
Similarly, if this sequence contains an odd number of symbols $1$, we can write $\omega_1\ldots\omega_n = 0^k 1 \omega'$ with $\omega'$ containing an even number of symbols $1$. As $f_{[0^k1]}(x) = f_0^{-k}(1-x)$, applying the previous argument, we have that $f_{[\omega_1\ldots\omega_n]}(x)=f_0^j(1-x)$ for some integer $j$.

This proves that the full forward orbit of $x$ by the IFS, $\cO^+(x)$, is contained in two sets $\{f_0^j(x)\colon j\in\bZ\}$ and $\{f_0^j(1-x)\colon j\in\bZ\}$, each of which has just two accumulation points: 0 and 1. This proves the proposition.
\end{proof}

\begin{proposition}
 $F$  has a unique measure of maximal entropy, which is nonhyperbolic.
\end{proposition}

%We have $f_0'(0) f_0'(1) f_1'(0) f_1'(1)=1$, hence there exists only one measure of maximal entropy supported on $\{0,1\}\times \Sigma_2$. 
\begin{proof}
	By Theorem \ref{the:1}, the measure $\mu_{\rm max}^{\rm ex}$ is unique and nonhyperbolic by our choice \eqref{eq:ufffa}. Hence, it is enough to prove that there cannot exist a measure of maximal entropy supported on $ \Sigma_2\times(0,1)$.

Arguing by contradiction, assume that such a measure exists, denote it by $\mu$. Its projection to $\Sigma_2$ must be the measure $\widehat\nu_{\rm max}$ of maximal entropy for $\sigma\colon\Sigma_2\to\Sigma_2$, that is,  the $(1/2,1/2)$-Bernoulli measure (recall Section~\ref{sec:twinsetc}). We shortly write $\nu=\widehat\nu_{\rm max}$.
The measure $\mu$ admits a disintegration, that is, there exists a family $\{\mu_\xi\colon\xi\in\Sigma_2\}$ of probabilities such that $\xi\mapsto\mu_\xi$ is measurable and every $\mu_\xi$ is supported on $\{\xi\}\times(0,1)$ and  satisfies
\[
	\mu (E)
	= \int \mu_\xi(E)\, d\nu(\xi)
\] 
for any measurable set $E$.
%with each measure $\mu_\omega$ living on $\{\omega\}\times(0,1)$. 
%The measure $\mu_\omega$ is defined for $\nu$-almost every $\omega$. 
With a slight lack of precision we will consider each $\mu_\xi$ as a measure on $(0,1)$.

To investigate what happens to $\mu_\xi$ under our dynamics, we will use the following result whose proof we postpone. Recall our notation $f_\omega^n\eqdef f_{\omega_{n-1}}\circ\ldots\circ f_{\omega_0}$.

\begin{lemma} \label{lemprop:randwalk}
For $\nu$-almost every $\omega\in\Sigma_2$, for every $\varepsilon>0$ and for every measure $\mu$ supported on $(0,1)$ we have
\[
\lim_{n\to\infty} \frac 1n \sum_{i=1}^n (f_{\omega}^i)_\ast \mu((\varepsilon, 1-\varepsilon)) =0.
\]
\end{lemma}

We postpone the proof of the above lemma to the following subsection.
Assuming that the above lemma was proven, we can now complete the proof of the theorem.
In particular, we can, for a $\nu$-generic $\omega$, apply Lemma \ref{lemprop:randwalk} to the measure $\mu = \mu_\omega$. Thus, recalling that $\mu$ is $F$-invariant, for every $n\ge1$ we have
\[\begin{split}
	\mu( \Sigma_2\times(\varepsilon, 1-\varepsilon)) 
	&= F_\ast\mu( \Sigma_2\times(\varepsilon, 1-\varepsilon)) 
	= \frac 1n \sum_{i=1}^n (F^i)_\ast\mu
		( \Sigma_2\times(\varepsilon, 1-\varepsilon))\\
	&= \frac 1n \sum_{i=1}^n \int (f_\omega^i)_\ast(\mu_\omega)((\varepsilon, 1-\varepsilon)) d\nu(\omega)\\
	&=\int 
		\frac 1n \sum_{i=1}^n (f_\omega^i)_\ast(\mu_\omega)((\varepsilon, 1-\varepsilon)) d\nu(\omega).
\end{split}\]
Now, by Lemma~\ref{lemprop:randwalk}, taking the limit $n\to\infty$ and apply the dominated convergence theorem, we obtain $\mu( \Sigma_2\times(\varepsilon, 1-\varepsilon)) =0$.
As $\varepsilon$ is arbitrary, this implies $\mu(\Sigma_2\times(0,1))=0$, contradiction to the fact that we assumed that $\mu$ was supported on $\Sigma_2\times(0,1)$. 
This finishes the proof of the proposition.
\end{proof}

This proves the theorem.
\end{proof}

\subsubsection{Random walks -- Proof of Lemma~\ref{lemprop:randwalk}}
To proof Lemma~\ref{lemprop:randwalk} we need to introduce several auxiliary objects in order to reduce it to well-known results.
Heuristically, a $\nu$-typical $\omega\in\Sigma_2$ can be treated as a random process with no memory, and then the dynamics generated by $f_\omega^i$ is given by a certain random walk. The result we will prove below is a version of a well-known statement that a random walk does not stay in any bounded region. 

For what we study below, we will consider the one-sided shift space $\Sigma_2^+$ only and by a slight abuse of notation continue to denote the $(1/2,1/2)$-Bernoulli measure on it by $\nu$. We consider a $\nu$-typical $\omega = (\omega_1\omega_2\ldots)\in \Sigma_2^+$ and interpret the values $\omega_i$ as random variables, with $\nu$ giving their joint distribution. That is, each $\omega_i$ takes values 0 and 1 with probabilities 1/2 each, independently of any other $\omega_j$'s. Denoting $\omega^n=\omega_1\ldots\omega_n$, let $\Omega_n$ be the $\sigma$-algebra generated by the cylinders $[\omega^1],\ldots,[\omega^n]$.

We first introduce the following auxiliary IFS of maps $g_0,g_1$. Let $\Omega=(0,1)\times\{+1,-1\}$ and define $g_0,g_1\colon\Omega\to\Omega$ by 
\[
	g_0(x,+1)\eqdef (f_0(x),+1),\quad
	g_0(x,-1)\eqdef (f_0^{-1}(x),-1),
\]
and 
\[
	g_1(x,+1)\eqdef (x,-1),\quad
	g_1(x,-1)\eqdef (x,+1).
\]
Consider the projections $\pi_1,\pi_2\colon\Omega\to(0,1)$ defined as follows
\[
	\pi_1(x,k) \eqdef 
	\begin{cases}
	x&\text{ if }k=+1,\\
	1-x&\text{ if }k=-1,
	\end{cases}
\quad\quad
	\pi_2(x,k) \eqdef x.
\]
One immediately checks that $\pi_1\circ g_i = f_i\circ \pi_1$, $i=0,1$, that is, the  original IFS $\{f_0,f_1\}$ on $(0,1)$ is a factor of the IFS $\{g_0, g_1\}$ on $\Omega$ under $\pi_1$. Note that, given $x\in (0,1)$, for every $n\ge1$ we have
\begin{equation}\label{eq:heree}
	(\pi_1\circ g_{[\omega^n]})(x,+1) = f_{[\omega^n]}(x),
\end{equation}
that is, we can consider  $(0,1)$ as $(0,1)\times \{+1\}$, apply the maps $g_i$ instead of $f_i$ and then project back the results by $\pi_1$ and get the same result as if we applied maps $f_i$ and never left $(0,1)$. Indeed, this is a consequence of our symmetry assumptions on $f_0$ and~\eqref{eqn:comm}.

To model the claimed random walk, we consider now $(0,1)\times \{+1\}$  instead of $(0,1)$ and apply the maps $g_i$ instead of $f_i$ and then project the results by $\phi^{-1}\circ\pi_2$. That is, let
\[
	R_i(x)
	=  R_i(x,\omega) 
	\eqdef (\phi^{-1} \circ\pi_2 \circ g_{[\omega^{i}]})(x,+1).
\]
This defines a random walk on $\bR$. The main aim of this section is to proof the following result.

\begin{lemma}\label{lemcor:last}
	For every probability measure $\mu$ on $\bR$ and for any bounded $A\subset \bR$, $\nu$-almost surely we have
\[
	\lim_{n\to\infty} \frac 1n\sum_{i=1}^{n}(R_i)_\ast\mu(A)=0.
\]
\end{lemma}

The above result now will provide the

\begin{proof}[Proof of Lemma~\ref{lemprop:randwalk}]
Note that, for any $\varepsilon>0$ we have 
\[
\pi_1^{-1}((\varepsilon,1-\varepsilon)) = \pi_2^{-1}((\varepsilon,1-\varepsilon))=(\varepsilon,1-\varepsilon)\times \{0,1\}.
\]
%The last step we need to make is to replace $\pi_2$ by $\pi_1$. 
Hence, by~\eqref{eq:heree}  for $A=\phi^{-1}((\varepsilon,1-\varepsilon))$  we have
\[
	(f_{[\omega^i]})_\ast\mu(\varepsilon,1-\varepsilon)
	= (\pi_1\circ g_{[\omega^i]})_\ast\mu(\varepsilon,1-\varepsilon)
	=(R_i)_\ast\mu(A).
\]
Applying now Lemma \ref{lemcor:last} implies the lemma, proving Lemma~\ref{lemprop:randwalk}.
\end{proof}

\subsubsection{Random walks -- Analysis of the random walk $R_i$}

This random process has a complicated behavior. We will introduce a sequence of   simpler auxiliary random processes which will help to prove Lemma~\ref{lemcor:last}.

Without loss of generality, we assume $\omega_1=0$.
Given $\omega$, let $n_i$, $i\ge0$, enumerate the positions at which in the sequence $\omega$ there appears the symbol $0$. With this notation, we have the following relation
\[\begin{split}
	&R_{n_i}(x,\omega)
	=\phi^{-1}\circ\,\pi_2\circ g_{[\omega^{n_i}]})(x,+1)\\
	&= \begin{cases}
		(\phi^{-1}\circ\pi_2\circ g_{[\omega^{n_{i-1}}]})(x,+1)+1
		&\text{ if }\#\{k\in\{n_{i-1},\ldots,n_i\}\colon\omega_k=1\}\text{ is even},\\
		(\phi^{-1} \circ\pi_2\circ g_{[\omega^{n_{i-1}}]})(x,+1)-1
		&\text{ if }\#\{k\in\{n_{i-1},\ldots,n_i\}\colon\omega_k=1\}\text{ is odd}.
	\end{cases}		
\end{split}\]
An elementary calculation shows that the number of 1's between any two consecutive 0's is even with probability $2/3$ and odd with probability $1/3$. 
That is, the random variable $n_i-n_{i-1}$ takes an even 
value with probability 1/3 and an odd 
value with probability 2/3, moreover this random variable is independent from $\Omega_{n_{i-1}}$. Considering then the subsequence $(n_i)_i$, we will now pass from the random walk $R_i$ to the following ``induced'' walk $S_i$, which is defined by
\[
	S_i(x)
	\eqdef (\phi^{-1} \circ\pi_2 \circ g_{[\omega^{n_i}]})(x,+1).
\]
The latter is a random walk on the real line composed by translations
\[
	S_i(x)
	=\begin{cases}
		S_{i-1}(x)+1&
			\text{ if } \#\{k\in\{n_{i-1},\ldots,n_i\}\colon\omega_k=1\}\text{ is even},\\
		S_{i-1}(x)-1&
			\text{ if }\#\{k\in\{n_{i-1},\ldots,n_i\}\colon\omega_k=1\}\text{ is odd},
	\end{cases}
\]
where each step being independently and identically distributed: in the same direction as the previous one with probability 2/3 
and in the opposite direction with probability 1/3 (with the convention that the `zeroth step' was in the positive direction). 

Since $S_i$ does not encode explicitly the information \emph{in which direction} the walk is moving (it does not carry the second coordinate), we will instead consider the following auxiliary walk.
Let  $U_i$ be a random walk on $\bR\times \{-1, +1\}$ given by
\[
U_i(x,j) \eqdef \begin{cases} 
(x+j,j) & \text{ with probability } 2/3,\\
(x-j,-j) & \text{ with probability } 1/3.
\end{cases} 
\]
which is just $S_i$ adding the information about the direction of the last step: there exists a measure preserving isomorphism under which the first coordinate of $U_i(x,+1)$ is equal to $S_i(x)$.

Recall that we want to show that the evolution of a measure under the application of the fiber maps of the IFS is eventually moving to the boundary of $(0,1)$, that is, to $\pm\infty$ for the walk lifted by $\phi^{-1}$ to $\bR$. For that reason, let us now consider an ``induced'' walk that only looks at times immediately after we moved in positive direction.
Let $V_i^+$ denote the random walk which is the first return of $U_i$ to $\bR \times \{+1\}$. That is,
\[
V_i^+(x) \eqdef \begin{cases}
x+1 & \text{ with probability } 2/3,\\
x-k, k=0,1,\ldots & \text{ with probability }2^k/3^{k+2}.
\end{cases}
\] 
At last we got an usual random walk.
Note that the walk $V_i^+$  is recurrent. Indeed, an elementary calculation gives that its expected displacement is zero, that is $\bE(V_i^+(x)-x)=0$.
Given $A\subset\bR$, let us define $\mathbbm1_A(y)=1$ if $y\in A$ and $\mathbbm1_A(y)=0$ otherwise.

The Chung and Erd\"os Theorem \cite{ChuErd:51} immediately implies

\begin{lemma}
For any bounded $A\subset \bR$, $\nu$-almost surely
\[
	\lim_{n\to\infty} \frac 1n \sum_{i=1}^n\mathbbm{1}_A(V_i^+(0))=0.
\]
\end{lemma}

\begin{proof}
	By~\cite[Theorem 3.1]{ChuErd:51}, for any $a,b\in\bZ$ we have 
\[
	\lim_{i\to\infty}\frac{\mathbb P(V_i^+(0)=a)}{\mathbb P(V_i^+(0)=b)}
	=1.
\]	
Hence, for any bounded set $A\subset\bZ$ for every $a\in A$ we have
\[
	\lim_{i\to\infty}\mathbb P(V_i^+(0)=a)\le \frac{1}{\lvert A\rvert},	
\]
where $\lvert A\rvert$ denotes the cardinality of $A$ and hence 
\[
	\lim_{i\to\infty}\mathbb P(V_i^+(0)\in A)
	=0.
\]
In particular, for every $a\in\bZ$ we  have 
\begin{equation}\label{eq:brades}
	\lim_{i\to\infty}\mathbb P(V_i^+(0)=a)
	=0.	
\end{equation}

Let $T_k$ denote the $k$th return time of $V_i^+(0)$ to $0$. Note that $T_k$ equals the sum of $k$ independent copies of $T_1$. 

\begin{claim}
	$\mathbb E(T_1)=\infty$.
\end{claim}

\begin{proof}
	By contradiction, assume that $\mathbb E(T_1)$ would be finite. Hence, by the strong law of large numbers, almost surely we would have $\lim_{n\to\infty}\frac1nT_n=\mathbb E(T_1)$. Hence, by Egorov's theorem, for any $\varepsilon\in(0,1)$ there would exist $N\ge1$ such that for all $k\ge N$ with probability at least $1-\varepsilon$ we would have $\frac1kT_k\le\mathbb E(T_1)+\varepsilon$.
This would imply 
\[
	\sum_{i=1}^{k(\mathbb E(T_1)+\varepsilon)}\mathbb P(V_i^+(0)=0)
	=\int d\omega
		\sum_{i=1}^{k(\mathbb E(T_1)+\varepsilon)}\mathbbm{1}_{\{0\}}(V_i^+(0))
	\ge(1-\varepsilon)k(\mathbb E(T_1)+\varepsilon),
\]
which would imply $\limsup_{i\to\infty} \mathbb P(V_i^+(0)=0)>0$. Contradiction with~\eqref{eq:brades}.
\end{proof}

Analogously, given $a\in\bZ$, let $T_1^a$ denote the first hitting time of $V_i^+(0)$ at $a$. Observe that the return times of $V_{i+T_1^a}^+(0)$ to $a$ have the same distribution as the return times of $V_i^+(0)$ to $0$. Hence, if $T_k^a$ denotes the $k$th return time of $V_i^+(0)$ to $a$, then analogously to the above claim we conclude $\mathbb E(T_1^a)=\infty$. 
Hence, by the strong law of large numbers almost surely we have $\frac1kT_k^a\to \infty$, which in turn implies 
\[
	0
	= \limsup_{k\to\infty}\frac{k}{T_k^a}
	= \limsup_{n\to\infty}\frac1n\sum_{i=1}^n\mathbbm{1}_{\{a\}}(V_i^+(0)).
\]

Writing now $\mathbbm{1}_A=\sum_{a\in A}\mathbbm{1}_a$, we obtain
\[
	\lim_{n\to\infty}\frac1n\sum_{i=1}^n\mathbbm{1}_{A}(V_i^+(0))
	=0
\]
almost surely.
This proves the lemma.
\end{proof}

Looking at the random walk $U_i$, we get the corresponding statement for the set $A\times \{+1\}$, that is
\[
	\lim_{n\to\infty} \frac 1n\sum_{i=1}^n\mathbbm{1}_{A\times\{+1\}}(U_i(0,+1))=0.
\]
The proof for $A\times \{-1\}$ is similar, only instead of $V_i^+$ we need to take the first return to $\bR \times \{-1\}$. Recalling that the projection of $U_i$ to the first coordinate is just $S_i$, we obtain the following corollary.

\begin{corollary} \label{cor:sometimes}
For any bounded $A\subset \bR$, $\nu$-almost surely
\[
\lim_{n\to\infty} \frac 1n \sum_{i=1}^n\mathbbm1_A(S_i(0))=0.
\]
\end{corollary}

Recall that $S_i$  takes into account only the steps of the initial walk when the symbol 0 appeared. Vaguely speaking, $S_i$ takes only into account whether $n_i-n_{i-1}$ is even or odd. Let $\Omega'$ be the $\sigma$-algebra generated by $S_i$. 
We will consider the following auxiliary random variable $d_i$ defined by
\[
d_i \eqdef 
\begin{cases}
n_i-n_{i-1} & \text{if }n_i - n_{i-1} \text{ is odd,}\\
n_i-n_{i-1}-1 &\text{if } n_i - n_{i-1} \text{ is even.}
\end{cases}
\]
Below we will argue that $d_i$ is independent of $\Omega'$.
The remaining information (needed to recover $\Omega = \bigcap_n \Omega_n$) is in exact values of $n_i-n_{i-1}$. Knowing $(S_i)_{i=1}^n$ and $(d_i)_{i=1}^n$, we can recover $(R_i)_{i=1}^n$. Note that $d_i$ being independent from $\Omega'$ means  that we can decompose the measure $\nu$ (the $(1/2,1/2)$-Bernoulli measure on $(\Sigma_2,\Omega)$) as $\mu_s \times \mu_d$, 
where $\mu_s$ is the distribution of $(S_i)$ and $\mu_d$ is the joint distribution of the i.i.d. random variables $d_i$. Hence then we can conclude that if for $\mu_s$-almost every realization $(S_i)$ for $\mu_d$-almost every realization $(d_i)$ an event holds, then it holds for $\nu$-almost every $\omega$.

\begin{lemma}\label{lem:independent}
The random variable $(d_i)$ is independent of $\Omega'$ and has finite expectation $\mathbb E(d_i)$.
\end{lemma}

\begin{proof}
Whether $n_i-n_{i-1}$ is even or odd, $d_i$ always has the same distribution
\[
	\mathbb P(d_i=2k+1| n_i-n_{i-1}\text{ even})
	= \mathbb P(n_i-n_{i-1}=2k+1| n_i-n_{i-1}\text{ odd}) = \frac{3}{4^{k+1}},
\]
which follows from an elementary calculation. 
Hence, in particular, the expected value of $d_i$ is finite. This proves the lemma.
\end{proof}

With the above, we now return to the random walk $R_i$. 

\begin{corollary}
For any bounded $A\subset \bR$, $\nu$-almost surely,
\[
	\lim_{n\to\infty} \frac 1n\sum_{i=1}^n \mathbbm1_A(R_i(0))
	=0.
\]
\end{corollary}

\begin{proof}
Note that $R_n(0)=R_{n_{i-1}}(0)=S_{i-1}(0)$ for every $n\in\{n_{i-1},\ldots, n_i-1\}$. Hence, we have
\[
	\frac 1{n_i} \sum_{k=0}^{n_i-1}\mathbbm1_A(R_k(0))
	=\frac 1 {n_i} \sum_{\ell=0}^{i-1} (n_{\ell+1}-n_\ell) \mathbbm1_A(S_\ell(0)).
\]
In particular, to show the claim, it is enough to consider the  specific subsequence from above proving that
\[
	0=\lim_{i\to\infty} \frac 1{n_i}\sum_{k=0}^{n_i-1}\mathbbm1_A(R_k(0))
	= \lim_{i\to\infty} \frac 1 {n_i} \sum_{\ell=0}^{i-1} 
	(n_{\ell+1}-n_\ell) \mathbbm1_A(S_\ell(0)).
\]
Thus, we have
\[\begin{split}
	&\limsup_{n\to\infty} \frac 1n \sum_{k=0}^{n-1} \mathbbm1_A(R_k(0)) 
	= \limsup_{i\to\infty} \frac 1 {n_i} \sum_{k=0}^{n_i-1}  \mathbbm1_A(R_k(0)) \\
	&= \limsup_{i\to\infty} \frac 1 {n_i} \sum_{\ell=0}^{i-1} (n_{\ell+1}-n_\ell)  \mathbbm1_A(S_\ell(0)) \\
	&= \limsup_{i\to\infty}\frac{i}{n_i} 
			\cdot  \frac{\sum_{\ell=0}^{i-1} \mathbbm1_A(S_\ell(0))}{i}
			\cdot \frac{\sum_{\ell=0}^{i-1} (n_{\ell+1}-n_\ell)  \mathbbm1_A(S_\ell(0))}{\sum_{\ell=0}^{i-1} \mathbbm1_A(S_\ell(0))}\\
	&\le \limsup_{i\to\infty}\frac{i}{n_i} 
			\cdot \limsup_{i\to\infty} \frac{\sum_{\ell=0}^{i-1} \mathbbm1_A(S_\ell(0))}{i}
			\cdot \limsup_{i\to\infty}\frac{\sum_{\ell=0}^{i-1} (n_{\ell+1}-n_\ell)  \mathbbm1_A(S_\ell(0))}{\sum_{\ell=0}^{i-1} \mathbbm1_A(S_\ell(0))}\\
	&=: L_1\cdot L_2\cdot L_3.		
\end{split}\]	

\begin{claim*}
	Almost surely, we have that $L_1$ and $L_3$  are finite and $L_2=0$.
\end{claim*}

With this claim and also using Lemma~\ref{lem:independent}, obtain that $\nu$-almost surely we have  $L_1\cdot L_2\cdot L_3=0$ and we conclude that $\nu$-almost surely
\[
	\limsup_{n\to\infty} \frac 1n \sum_{k=0}^{n-1} \mathbbm1_A(R_k(0)) 
	=0,
\]
proving the corollary. What remains is to prove the claim.

To estimate these latter terms, first observe that almost surely
\[
	\frac{n_i}{i}
	= \frac{(n_i-n_{i-1})+\ldots+(n_1-0)}{i}	
	\to \mathbb E(n_i-n_{i-1}).
\]
Hence
\[
	\mathbb E(n_i-n_{i-1})
	= \sum_{i\ge1}\Big(d_i\mathbb P(n_i-n_{i-1}\text{ even})
								+(d_i+1)\mathbb P(n_i-n_{i-1}\text{ odd})\Big)
	\ge \mathbb E(d_i)								
\]
and therefore $L_1\le (\mathbb E(d_i))^{-1}		<\infty$.
Moreover, this calculation also gives
\begin{equation}\label{eq:L1second}
	\mathbb E(n_i-n_{i-1})
	\le \mathbb E(d_i)+1.
\end{equation}

By Corollary~\ref{cor:sometimes}, we have $L_2=0$.

To estimate $L_3$, first observe that, as the expected value of $d_i$ is finite, if we fix $(S_i)$ then, by the law of large numbers, almost surely the average of $\{d_i\colon S_i(0)\in A\}$ and the average of $\{d_i\colon  S_i(0)\notin A\}$ converge to the same limit $\mathbb E(d_i)$. 
Thus, almost surely 
\[
	\lim_{i\to\infty}\frac{\sum_{\ell=0}^{i-1} (n_{\ell+1}-n_\ell)  \mathbbm1_A(S_\ell(0))}
			{\sum_{\ell=0}^{i-1} \mathbbm1_A(S_\ell(0))}
	=\lim_{i\to\infty}\frac{	\sum_{\ell=0}^{i-1} (n_{\ell+1}-n_\ell)}
									{i}
	\le \mathbb E(d_i)+1,			
\]
where the latter follows from~\eqref{eq:L1second}.
Thus, $L_3$ is finite. This proves the claim.
\end{proof}

The statement for random walk starting from 0 can be generalized to any starting distribution which allows us to finally prove Lemma~\ref{lemcor:last}.

\begin{proof}[Proof of Lemma~\ref{lemcor:last}]
Fix $\mu$. For any $\varepsilon>0$ we can find some $N$  such that $\mu([-N,N])>1-\varepsilon$. As $\phi^{-1}\circ\pi_2\circ g_i$ is a translation, if $R_n(0) \notin B_N(A)$ then $(R_n)_*\mu(A) < \varepsilon$. Hence,

\[
\limsup_{n\to\infty} \frac 1n (R_n)_*\mu(A)\leq \varepsilon + \lim_{n\to\infty} \frac 1n \mathbbm1_{B_N(A)}(R_n(0))=\varepsilon.
\]
Passing with $\varepsilon$ to 0 ends the proof.
\end{proof}

\bibliographystyle{plain}

\begin{thebibliography}{10}

\bibitem{BocBonDia:16}
Jairo Bochi, Christian Bonatti, and Lorenzo~J. D\'iaz.
\newblock Robust criterion for the existence of nonhyperbolic ergodic measures.
\newblock {\em Comm. Math. Phys.}, 344(3):751--795, 2016.

\bibitem{BocBonGel:}
Jairo Bochi, Christian Bonatti, and Katrin Gelfert.
\newblock Dominated {Pesin} theory: Convex sum of hyperbolic measures.
\newblock {\em To appear in: Israel J. Math.}

\bibitem{BonDiaVia:05}
Christian Bonatti, Lorenzo~J. D\'iaz, and Marcelo Viana.
\newblock {\em Dynamics beyond uniform hyperbolicity}, volume 102 of {\em
  Encyclopaedia of Mathematical Sciences}.
\newblock Springer-Verlag, Berlin, 2005.
\newblock A global geometric and probabilistic perspective, Mathematical
  Physics, III.

\bibitem{ChuErd:51}
Kai~Lai Chung and Paul Erd\"os.
\newblock Probability limit theorems assuming only the first moment. {I}.
\newblock {\em Mem. Amer. Math. Soc.,}, No. 6:19, 1951.

\bibitem{Cro:11}
Sylvain Crovisier.
\newblock Partial hyperbolicity far from homoclinic bifurcations.
\newblock {\em Adv. Math.}, 226(1):673--726, 2011.

\bibitem{DiaEstRoc:16}
Lorenzo~J. D\'iaz, Salete Esteves, and Jorge Rocha.
\newblock Skew product cycles with rich dynamics: from totally non-hyperbolic
  dynamics to fully prevalent hyperbolicity.
\newblock {\em Dyn. Syst.}, 31(1):1--40, 2016.

\bibitem{DiaGel:12}
Lorenzo~J. D\'iaz and Katrin Gelfert.
\newblock Porcupine-like horseshoes: transitivity, {L}yapunov spectrum, and
  phase transitions.
\newblock {\em Fund. Math.}, 216(1):55--100, 2012.

\bibitem{DiaGelRam:13}
Lorenzo~J. D\'iaz, Katrin Gelfert, and {Micha\l} Rams.
\newblock Almost complete {L}yapunov spectrum in step skew-products.
\newblock {\em Dyn. Syst.}, 28(1):76--110, 2013.

\bibitem{DiaGelRam:14}
Lorenzo~J. D\'iaz, Katrin Gelfert, and {Micha\l} Rams.
\newblock Abundant rich phase transitions in step-skew products.
\newblock {\em Nonlinearity}, 27(9):2255--2280, 2014.

\bibitem{DiaGelRam:17a}
Lorenzo~J. D\'iaz, Katrin Gelfert, and {Micha\l} Rams.
\newblock Nonhyperbolic step skew-products: ergodic approximation.
\newblock {\em Ann. Inst. H. Poincar\'e Anal. Non Lin\'eaire},
  34(6):1561--1598, 2017.

\bibitem{DiaGelRam:17}
Lorenzo~J. D\'iaz, Katrin Gelfert, and {Micha\l} Rams.
\newblock Topological and ergodic aspects of partially hyperbolic
  diffeomorphisms and nonhyperbolic step skew products.
\newblock {\em Tr. Mat. Inst. Steklova}, 297(Poryadok i Khaos v Dinamicheskikh
  Sistemakh):113--132, 2017.

\bibitem{DiaHorRioSam:09}
Lorenzo~J. D\'iaz, Vanderlei Horita, Isabel Rios, and Martin Sambarino.
\newblock Destroying horseshoes via heterodimensional cycles: generating
  bifurcations inside homoclinic classes.
\newblock {\em Ergodic Theory Dynam. Systems}, 29(2):433--474, 2009.

\bibitem{DiaMar:15}
Lorenzo~J. D\'iaz and Tiane Marcarini.
\newblock Generation of spines in porcupine-like horseshoes.
\newblock {\em Nonlinearity}, 28(11):4249--4279, 2015.

\bibitem{Gau:92}
Andrea Gaunersdorfer.
\newblock Time averages for heteroclinic attractors.
\newblock {\em SIAM J. Appl. Math.}, 52(5):1476--1489, 1992.

\bibitem{GelKwi:}
Katrin Gelfert and Dominik Kwietniak.
\newblock On density of ergodic measures and generic points.
\newblock {\em To appear in: Ergodic Theory Dynam. Systems}.

\bibitem{GorPes:17}
Anton Gorodetski and Yakov Pesin.
\newblock Path connectedness and entropy density of the space of hyperbolic
  ergodic measures.
\newblock In {\em Modern theory of dynamical systems}, volume 692 of {\em
  Contemp. Math.}, pages 111--121. Amer. Math. Soc., Providence, RI, 2017.

\bibitem{IlyShi:17}
Yulij~S. Ilyashenko and Ivan Shilin.
\newblock Attractors and skew products.
\newblock In {\em Modern theory of dynamical systems}, volume 692 of {\em
  Contemp. Math.}, pages 155--175. Amer. Math. Soc., Providence, RI, 2017.

\bibitem{Kan:84}
Ittai Kan.
\newblock Open sets of diffeomorphisms having two attractors, each with an
  everywhere dense basin.
\newblock {\em Bull. Amer. Math. Soc. (N.S.)}, 31(1):68--74, 1994.

\bibitem{KatHas:95}
Anatole Katok and Boris Hasselblatt.
\newblock {\em Introduction to the modern theory of dynamical systems},
  volume~54 of {\em Encyclopedia of Mathematics and its Applications}.
\newblock Cambridge University Press, Cambridge, 1995.
\newblock With a supplementary chapter by Katok and Leonardo Mendoza.

\bibitem{LedWal:77}
Fran\c{c}ois Ledrappier and Peter Walters.
\newblock A relativised variational principle for continuous transformations.
\newblock {\em J. London Math. Soc. (2)}, 16(3):568--576, 1977.

\bibitem{LepOliRio:11}
Renaud Leplaideur, Krerley Oliveira, and Isabel Rios.
\newblock Equilibrium states for partially hyperbolic horseshoes.
\newblock {\em Ergodic Theory Dynam. Systems}, 31(1):179--195, 2011.

\bibitem{LinOlsSte:78}
Joram Lindenstrauss, Gunnar~H. Olsen, and Yaki Sternfeld.
\newblock The {P}oulsen simplex.
\newblock {\em Ann. Inst. Fourier (Grenoble)}, 28(1):vi, 91--114, 1978.

\bibitem{New:80}
Sheldon~E. Newhouse.
\newblock Lectures on dynamical systems.
\newblock In {\em Dynamical systems ({B}ressanone, 1978)}, pages 209--312.
  Liguori, Naples, 1980.

\bibitem{RamSiq:17}
Vanessa Ramos and Jaqueline Siqueira.
\newblock On equilibrium states for partially hyperbolic horseshoes: uniqueness
  and statistical properties.
\newblock {\em Bull. Braz. Math. Soc. (N.S.)}, 48(3):347--375, 2017.

\bibitem{RioSiq:18}
Isabel Rios and Jaqueline Siqueira.
\newblock On equilibrium states for partially hyperbolic horseshoes.
\newblock {\em Ergodic Theory Dynam. Systems}, 38(1):301--335, 2018.

\bibitem{Sig:74}
Karl Sigmund.
\newblock On dynamical systems with the specification property.
\newblock {\em Trans. Amer. Math. Soc.}, 190:285--299, 1974.

\bibitem{Sig:77}
Karl Sigmund.
\newblock On the connectedness of ergodic systems.
\newblock {\em Manuscripta Math.}, 22(1):27--32, 1977.

\bibitem{TahYan:}
Ali Tahzibi and Jiagang Yang.
\newblock Strong hyperbolicity of ergodic measures with large entropy.
\newblock {\em Preprint {\tt arXiv:1606.09429}, To appear in: {Trans. Amer.
  Math. Soc.}}

\bibitem{Tak:94}
Floris Takens.
\newblock Heteroclinic attractors: time averages and moduli of topological
  conjugacy.
\newblock {\em Bol. Soc. Brasil. Mat. (N.S.)}, 25(1):107--120, 1994.

\bibitem{Wal:82}
Peter Walters.
\newblock {\em An introduction to ergodic theory}, volume~79 of {\em Graduate
  Texts in Mathematics}.
\newblock Springer-Verlag, New York-Berlin, 1982.

\end{thebibliography}

\end{document}